\documentclass[12pt]{amsart}

\usepackage[margin=1in]{geometry}
\usepackage{amscd,amssymb, amsmath,amsfonts, wasysym, mathrsfs, enumitem, mathtools,hhline,color}%enumerate
\usepackage{xcolor}
\usepackage{tikz-cd}
\usepackage{tikz}
\usetikzlibrary{matrix,arrows,decorations.markings}
\usepackage{graphicx}
\usepackage[all, cmtip]{xy}
\usepackage{url}
\usepackage[normalem]{ulem}

\definecolor{hot}{RGB}{65,105,225}

\usepackage[pagebackref=true,colorlinks=true, linkcolor=hot ,  citecolor=hot, urlcolor=hot]{hyperref}%make all the references and links clickable
\usepackage{caption}
\usepackage{subcaption}%Side by side figures
\theoremstyle{plain}
\newtheorem{theorem}{Theorem}[section]
\newtheorem{prop}[theorem]{Proposition}

\newtheorem{cor}[theorem]{Corollary}

\newtheorem{lemma}[theorem]{Lemma}

\theoremstyle{definition}

\newtheorem{defn}[theorem]{Definition}
\newtheorem{que}[theorem]{Question}
\newtheorem{remark}[theorem]{Remark}

\newcommand\sO{{\mathcal O}}

\newcommand\sH{{\mathcal H}}

\newcommand\sP{{\mathcal P}}

\newcommand\sF{{\mathcal F}}

\newcommand\sG{{\mathcal G}}
\newcommand\sM{{\mathcal M}}
\newcommand\sN{{\mathcal N}}

\newcommand\sV{{\mathcal V}}
\newcommand\sL{\mathcal{L}}

\def\R{\mathbb{R}}

\newcommand\Q{{\mathbb{Q}}}
\newcommand\Z{{\mathbb{Z}}}

\newcommand\C{{\mathbb{C}}}

\renewcommand\P{{\mathbb{P}}}

\newcommand{\Art}{S_0}

\newcommand{\pH}{\,^p\sH}
\newcommand{\wt}[1]{\widetilde{#1}}
\newcommand{\ul}[1]{\underline{#1}}

\newcommand{\green}[1]{{\color{olive}#1}}

\DeclareMathOperator{\id}{id} 
\DeclareMathOperator{\Mel}{Mel}

\DeclareMathOperator{\Tors}{Tors}
\DeclareMathOperator{\Perv}{Perv}

\DeclareMathOperator{\rat}{rat}

\DeclareMathOperator{\Hom}{Hom}
\DeclareMathOperator{\Ext}{Ext}

\def\sF{\mathcal{F}}

\title[Alexander modules]{Alexander modules, Mellin transformation and variations of mixed Hodge structures}
\begin{document}

\author{Eva Elduque}
\address{Department of Mathematics, University of Michigan-Ann Arbor, 530 Church St, Ann Arbor, MI 48109, USA.}
\email {elduque@umich.edu}\urladdr{http://www-personal.umich.edu/~elduque}

%\author{Christian Geske}
%\address{Department of Mathematics, Northwestern University, 2033 Sheridan Rd, Evanston, IL 60208, USA.}
%\email {christian.geske@northwestern.edu}

\author{Mois\'es Herrad\'on Cueto}
\address{Department of Mathematics, Louisiana State University, 303 Lockett Hall, Baton Rouge, LA 70803, USA.}
\email {moises@lsu.edu}\urladdr{http://www.math.lsu.edu/~moises}

\author{Lauren\c{t}iu Maxim}
\address{Department of Mathematics, University of Wisconsin-Madison, 480 Lincoln Drive, Madison WI 53706-1388, USA.}
\email {maxim@math.wisc.edu}\urladdr{https://www.math.wisc.edu/~maxim/}

\author{Botong Wang}
\address{Department of Mathematics, University of Wisconsin-Madison, 480 Lincoln Drive, Madison WI 53706-1388, USA.}
\email {wang@math.wisc.edu}\urladdr{http://www.math.wisc.edu/~wang/}

\renewcommand{\shortauthors}{Elduque, Herrad\'on Cueto, Maxim and Wang}



\date{\today}

\begin{abstract}
To any complex algebraic variety endowed with a morphism to a complex affine torus we associate multivariable cohomological Alexander modules, and define natural mixed Hodge structures on their maximal Artinian submodules. The key ingredients of our construction are Gabber-Loeser's Mellin transformation and Hain-Zucker's work on unipotent variations of mixed Hodge structures. As applications, we prove the quasi-unipotence of monodromy, we obtain upper bounds on the sizes of the Jordan blocks of monodromy, and we explore the change in the Alexander modules after removing fibers of the map. We also give an example of a variety whose Alexander module has non-semisimple torsion.
\end{abstract}

\maketitle

%%%%%%%%%%%%%%%%%%%%%%%%%%%%%%%%%%%%%%%%%%%%%%%%%%%%

\section{Introduction}

%\green{[[[Feel free to remove the green color anywhere]]]} {\color{red} [[[Same for text in red color]]]}

%\green{\textbf{Mois\'es}: At this point, I've more or less lost track of our goal. Two meetings ago it seemed like we wanted to try to make the paper under 20 pages, so I was trying to save space wherever possible. Now, I don't know which arguments are considered too short (``algebra magic'') vs too long (straightforward and/or trivial), so I feel aimless in which arguments need to be shorter and which need to be longer. Again, feel free to make any changes or undo the changes I've made. I will limit myself to reading through the new changes, but I won't proactively make any changes unless/until I understand what we're supposed to do.

%Something else that came up two meetings ago was that the proofs in section 6 are too long and that we shouldn't have spelled out standard arguments. If this is still true, please read through it and remove anything that is well-known.}

\subsection{Setup}\label{sec:setup} The aim of this note is to investigate Hodge-theoretic aspects of multivariable cohomological Alexander modules associated to complex algebraic varieties endowed with morphisms to complex affine tori. We also provide some geometric applications, and indicate several methods of computation.

Let $T$ be a complex affine torus of dimension $n$, with a universal covering $\pi: \wt T\to T$, i.e., $T \cong (\C^*)^n$ and $\wt T\cong \C^n$. We fix a base point $\tilde{b}\in \wt T$, and let $b=\pi({\tilde{b}})$.   

Let $X$ be a complex algebraic variety, and let $f: X\to T$ be an algebraic map. Consider the fiber product $\widetilde{X}\coloneqq X\times_T \wt T$, with projections $p$ and $\tilde{f}$, as in the following diagram 
\[
\xymatrix{
\widetilde{X}\ar[r]^{\tilde{f}}\ar[d]_{p}&\wt T\ar[d]^{\pi}\\
X\ar[r]^{f}&T.
}
\]
In particular, $p:\widetilde{X} \to X$ is a $\Z^n$-covering map, and $\tilde{f}$ is the pullback of $f$ by $\pi: \wt T\to T$. 

For any $x\in T$, we denote $f^{-1}(x)$ by $F_x$, and for any $\tilde{x}\in \wt T$, we denote $\tilde{f}^{-1}(\tilde{x})$ by $F_{\tilde{x}}$. Note that if $x=\pi(\tilde{x})$, the covering map $p$ induces an isomorphism $F_{\tilde{x}}\cong F_x$.

%\green{I think for the sake of writing either we write ``fix an isomorphism $\pi_1(T)\cong \Z^n$ and let $t_1,\ldots ,t_n$ be the generators'', or we decide to treat $\pi_1(T)$ ina  coordinate-free way. I don't care which way we go but I think it should be one or the other.} 
Let $A=\Q[\pi_1(T)]$ and let
\[
\sL_T
\coloneqq \pi_!\underline{\Q}_{\wt T}
\]
be the tautological local system on $T$, seen as a local system of rank one free $A$-modules. The action of $A$ is defined by letting $\pi_1(T)$ act as deck transformations of $\wt T$ (which we convene is a right action in this paper). 
%\green{\cong}\Q[\Z^n]
%, which is a Laurent polynomial ring in $n$ variables $t_1,\ldots, t_n$.
 If we choose  an isomorphism $\pi_1(T)\cong \Z^n$, then $A\cong \Q[\Z^n]$ is isomorphic to a Laurent polynomial ring in $n$ variables $t_1,\ldots, t_n$. In these notations, the tautological local system $\sL_T$ is defined by letting  the standard generators $\alpha_i$ ($i=1,\ldots,n$) of $\pi_1(T)$ act by multiplication by the corresponding variable $t_i$.
%\green{Fix this later:} Explicitly, the standard generators $\alpha_i$ ($i=1,\ldots,n$) of $\pi_1(T)$ act by multiplication by the corresponding variable $t_i$. 
Let 
\begin{equation}\label{eq_p}
\sL_X \coloneqq f^*\sL_T \simeq p_! \underline{\Q}_{\widetilde{X}},
\end{equation}
be the induced local system on $X$
(where the isomorphism follows by proper base change). 

In the above notations, we can now introduce the main object of study in this paper.

\begin{defn}\label{def1}
For any non-negative integer $i$, the $i$-th \emph{cohomological Alexander module} of $X$ with respect to $f$ is the $A$-module $H^i(X, \sL_X)$. 
\end{defn}

The motivation for terminology comes from the fact that the corresponding homology modules $H_i(X,\sL_X)$ can be identified with the multivariable homology Alexander modules $H_i(\widetilde{X},\Q)$ of the pair $(X,f)$, with the module structure induced by the deck group action. Up to replacing $\sL_X$ by its $A$-dual, the cohomological Alexander modules may be related to the homological ones by the Universal Coefficient spectral sequence (see, e.g., \cite[Section 2.3]{DM07}), which in the case $n=1$ simplifies into a short exact sequence, see \cite[Remark 2.3.4]{EGHMW}. We note that the modules $H^*(X,\sL_X)$ are not isomorphic to the cohomology of $\widetilde X$ in general. Indeed, the former are computed as the hypercohomology of $p_!\underline{\Q}_{\widetilde{X}}$, and the latter as the hypercohomology of $p_*\underline{\Q}_{\widetilde{X}}$.

Multivariable (co)homological Alexander modules are typically studied through their support loci, see, e.g., \cite{DM07}, \cite{Li92}, \cite{Bu15}, \cite{LM17}. In this note, we investigate properties of their maximal Artinian submodules.

%Recall that, given a finitely generated $A$-module $M$, there exists a unique maximal Artinian $A$-submodule of $M$, which we denote by $\Art M$. In other words, $\Art M$ is the largest $A$-module that is supported on zero-dimensional subschemes of $\mathrm{Spec}(A)$. The submodule $\Art M$ is a finite dimensional $\Q$-vector space, which is naturally a representation of $\pi_1(T)=\Z^n$. For example, if $n=1$, then $A\cong \Q[t^{\pm 1}]$, and $\Art M$ is equal to the torsion submodule of $M$.

Recall that for a finitely generated $A$-modules, being Artinian is equivalent to being a finite dimensional $\Q$-vector space, and to having zero-dimensional support in $\operatorname{Spec}(A)$. Every finitely generated $A$-module $M$ has a maximal Artinian submodule, which we denote $\Art M$. For example, if $n=1$, then $A\cong \Q[t^{\pm 1}]$, and $\Art M$ is equal to the torsion submodule of $M$.

In this paper, we study $\Art H^i(X,\sL_X)$ for a complex algebraic variety $X$ (not necessarily smooth) endowed with an algebraic map $f\colon X\to T$ to an $n$-dimensional complex affine torus. We define natural and functorial mixed Hodge structures on the Artinian modules $\Art H^i(X, \sL_X)$. Our construction uses Saito's theory of mixed Hodge modules. The structure we obtain here depends on the choice of a base point in the universal covering $\wt T$ of $T$. However, different base points give rise to (non-canonically) isomorphic mixed Hodge structures on $\Art H^i(X, \sL_X)$. Because of this dependence on a base point, the focus of this paper is on exploring consequences of the construction and existence of the mixed Hodge structure we describe (see Corollaries~\ref{van} and \ref{iso}, Propositions~\ref{prop:removeFibers} and \ref{ss} and the example in Section~\ref{sec:non-ss}), and not so much on the mixed Hodge structure itself.

\begin{remark}[Maximal Artinian submodules of the homology and cohomology Alexander modules]\label{rem:homvscohom}
It should be noted that when $n>1$, there is no relation between $\Art H^j(X,\sL_X)$ and $\Art H_i(X,\sL_X)$, for any $i,j$. This is a consequence of a fact of commutative algebra: (up to replacing $\sL_X$ by its $A$-dual) the cohomology and homology of $\sL_X$ are computed by $A$-dual complexes of $A$-modules. If $n>1$, there is no relation between the maximal Artinian submodule of the cohomology of such a complex and that of its dual.

Examples are readily available: as soon as $n>\dim X$, $\Art H^j(X,\sL_X)=0$ (see Corollary~\ref{van}(b)), yet the homology counterpart need not vanish, and indeed $\Art H_0(X,\sL_X)\cong \Q$ if $\widetilde{X}$ is connected. Thus, the mixed Hodge structure in the present paper does not yield a mixed Hodge structure on (the maximal Artinian submodule of) the homology of $\widetilde{X}$, since this coincides with $\Art H_i(X,\sL_X)$.

On the other hand, if $n=1$, it was proved in \cite[Proposition 2.4.1]{EGHMW} that the maximal Artinian submodules of the homology and, resp., cohomology Alexander modules are canonically dual (up to a shift in degrees).
\end{remark}

%%%%%%%%%%%%%%%%%%%%%%%%%

\subsection{Statement of results}

Our main result can be formulated as follows:

\begin{theorem}\label{thm1}
Fixing $\wt T$, $\pi$, and a base point $\tilde{b} \in \wt T$, the maximal Artinian submodule $\Art H^i(X, \sL_X)$ has a natural mixed Hodge structure (MHS)  with the following properties.
\begin{enumerate}
\item If $b=\pi({\tilde{b}})$ is a general point of $f$, then there is a natural map 
\[
\iota: \Art H^i(X, \sL_X)\to H^{i-n}\big(F_{b}, \Q\big),
\]
which is an  injective morphism of MHS. 

\item For a general point $x\in T$, $\Art H^i(X, \sL_X)$ is non-canonically isomorphic to a sub-MHS of $H^{i-n}\big(F_x, \Q\big)$. 

\item Given any commutative triangle of algebraic maps
%\[
%\xymatrix{
%X\ar[rr]^{\phi}\ar[dr]_{f}&&Y\ar[dl]^{g}\\
%&T&
%}
%\]
\[
\begin{tikzcd}[row sep = 1em]
X \arrow[rr,"\phi"]\arrow[dr,"f"'] &&
Y \arrow[dl,"g"]\\
& T
\end{tikzcd}
\]
the induced map $\Art H^i(Y, \sL_Y)\to \Art H^i(X, \sL_X)$ is a homomorphism of MHS. 
\item The $\pi_1(T)$-action on $\Art H^i(X, \sL_X)$  is quasi-unipotent.
Moreover, if $\sigma$ is any monodromy action on $\Art H^i(X, \sL_X)$, then
\[
\log\sigma^N: \Art H^i(X, \sL_X)\to \Art H^i(X, \sL_X)(-1)
\]
is a map of MHS, where $N$ is positive integer such that $\sigma^N$ is unipotent, and $(-1)$ denotes the $-1$st Tate twist. 
\item If the $\pi_1(T)$-action on $\Art H^i(X, \sL_X)$ is semi-simple, then the MHS on $\Art H^i(X, \sL_X)$ is independent of the choice of the base point $\tilde{b}$. Otherwise, different base points give rise to MHS that are non-canonically isomorphic. 
\end{enumerate}
Here, we call a point $x \in T$ general (or generic) if it is contained in a Zariski open dense subset of $T$ over which $f$ is a topologically locally trivial fibration.
\end{theorem}

\begin{remark}
Note that property (1) of Theorem \ref{thm1} may fail if the genericity assumption on $b$ is dropped. For example, let $X=T\setminus \{b\}$ and let $f:X \to T$ be the inclusion map. Then $\Art H^n(X, \sL_X)\cong \Q$, but $F_b=\emptyset$. 
\end{remark}

Here are some consequences of Theorem \ref{thm1}. 
\begin{cor}\label{van}\label{cor_dominant}
Let $d$ be the dimension of a general fiber $F_x$ of $f$. 
\begin{enumerate}[label=(\alph*)]
\item For all $i<n$ and $i>n+2d$, we have $\Art H^i(X, \sL_X)\cong 0$.
\item If $f$ is not dominant, or more generally, if for two generic points $x, y\in T$, $H^{i-n}(F_x, \Q)$ and $H^{i-n}(F_y, \Q)$ do not share a nontrivial common sub-MHS, then $$\Art H^i(X, \sL_X)=0.$$
\item For any $\sigma \in \pi_1(T)$, let $N$ be a positive integer such that the action of $\sigma^N$ on $\Art H^i(X, \sL_X)$ is unipotent. Then the nilpotence index of $\sigma^N-1$ is at most $1+\min\{i-n, 2d-i+n\}$. In other words, after field extension to $\bar{\Q}$, every Jordan block of the action $\sigma$ has size at most $1+\min\{i-n, 2d-i+n\}$. If $X$ is smooth, or more generally a general fiber $F_x$ is smooth, then nilpotence index of $\sigma^N-1$ is at most 
$\min\left\{\left\lceil\frac{i-n+1}{2}\right\rceil, d-\left\lfloor\frac{i-n-1}{2}\right\rfloor\right\}$.
\item If $F_x$ is a smooth complete algebraic variety, or equivalently the map $f$ is generically smooth and proper, then the MHS on $\Art H^i(X, \sL_X)$ is pure of weight $i-n$. 
\end{enumerate}
\end{cor}

It should be noted that in Corollary~\ref{van}(c) we obtain the same bound that was obtained in \cite[Corollary 7.4.2]{EGHMW} in the case where $T$ is one dimensional and $X$ is smooth.

As a corollary of the construction of the map $\iota$ from Theorem \ref{thm1}(1), we also get the following.

\begin{cor}\label{iso}
If the algebraic map $f\colon X \to T$ is a topologically locally trivial fibration, then the cohomological Alexander modules $H^i(X, \sL_X)$ are Artinian, and the map $\iota$ is an isomorphism of MHS.
\end{cor}

We also explore the change in the Alexander modules after removing fibers of the map $f$. The following result is useful for explicit computations of Alexander modules (see, e.g., Section \ref{sec:non-ss}). Its proof does not involve mixed Hodge structures. 

\begin{prop}\label{prop:removeFibers}
Let $X$ be an algebraic variety with an algebraic map $f\colon X\to T$. Let $Z$ be a proper closed subset of $T$, let $U=T\setminus Z$ and $Y=f^{-1}(U)\xhookrightarrow{\tilde\jmath} X$. Let $\sL_{Y}\coloneqq(f \circ \tilde\jmath)^*\sL_T$ be the corresponding $A$-local system on $Y$ induced via pullback from the tautological local system on $T$. Then we have the following:
\begin{enumerate}
\item The isomorphism $\tilde{\jmath}^*{\sL_X}\cong \sL_{Y}$ induces an injection:
\[
S_0H^i(X,\sL_X)\hookrightarrow S_0H^i(Y,\sL_{Y}).
\]
\item If $Z$ is contained in the open set over which $f$ is a fibration, then the above injection is an isomorphism. 

\item Furthermore, suppose $Z=\{x\}$ is a point contained in the open set over which $f$ is a fibration. Then $H^i(Y,\sL_{Y})\cong H^i(X,\sL_X)\oplus A^{b_{i-2n+1}(F_x)}$, where $b_{i-2n+1}(F_x)$ denotes the Betti number of the fiber $F_x=f^{-1}(x)$.  %\green{This has no reference to the map anymore, so it's less precise, but in return the proof of Corollary 1.8 is much shorter if we only look at the isomorphism class}

%Furthermore, suppose $Z=\{x\}$ is a point contained in the open set over which $f$ is a fibration. Let $b_k(F_x)$ denote the $k$-th Betti number. Then there is a short exact sequence
%\[
%0\to H^i(X,\sL_X)\to H^i(Y,\sL_{Y}) \to A^{b_{i-2n+1}(F_x)}\to 0.
%\]
\item Suppose that the image of $f$ is an open set $B\subseteq T$, and that $f$ is a locally trivial fibration over $B$ with fiber $F$. Further, suppose that $T\setminus B$ is a hypersurface. Let $\overline{H}{}^i\subseteq H^i(F,\Q)$ be the subspace that is fixed by the monodromy of $f$ for every loop in $B$ whose image in $T$ is trivial. Then:
\[
S_0H^i(X,\sL_X)\cong \overline{H}{}^{i-n}.
\]  
\end{enumerate}
\end{prop}

In the case when $T=\C^*$, Proposition~\ref{prop:removeFibers} has the following homological counterpart.

\begin{cor}[Homology Alexander modules, case $n=1$]\label{cor:homologyOpen}
Let $X$ be an algebraic variety with an algebraic map $f:X\rightarrow \C^*$. Let $x\in  \C^*$, let $Y=X\setminus F_x\xhookrightarrow{\tilde\jmath} X$, and let $\sL_Y\coloneqq (f\circ\tilde{\jmath})^*\sL_{\C^*}$. Let $F$ be the generic fiber of $f$, and let $b_k(F)$ be its $k$-th Betti number. 
\begin{enumerate}
\item The inclusion $\tilde{\jmath}: Y\hookrightarrow X$ induces an open inclusion of infinite cyclic covers, $ \widetilde Y\hookrightarrow \widetilde X$, which in turn induces a homomorphism (see \cite[Remark 2.2.3]{EGHMW}):
\[
\Tors_A H_i(Y,\sL_Y)\cong \Tors_A H_i(\widetilde Y,\Q)\twoheadrightarrow  \Tors_A H_i(\widetilde X,\Q)\cong\Tors_A H_i(X,\sL_X).
\]
This homomorphism is surjective.
\item If $x$ is contained in the open set of $\C^*$ over which $f$ is a fibration, then the above surjection is an isomorphism. Furthermore, $H_i(Y,\sL_Y)\cong  H_i(X,\sL_X) \oplus A^{b_{i-1}(F_x)}$.%\green{Removed any reference to maps here}
\item 

%Suppose that the image of $f$ is an open set $B\subseteq\C^*$, where $B$ is the result of removing $m$ points $p_1,\ldots,p_m$ from $\C^*$. Suppose that $f$ is a locally trivial fibration over $B$ with fiber $F$. Let $\gamma_i\in\pi_1(B)$ be a loop around $p_i$. Let $V_k$ be the image of the endomorphism $\gamma_k-\Id$ of $H_i(F,\Q)$, where $\gamma_k$ is seen as the monodromy automorphism of $H_i(F,\Q)$ that it induces.  Then,
%$$\Tors_A H_i(X,\sL_X)\cong H_i(F,\Q)/\left(\oplus_{k=1}^m V_k\right)$$.

Suppose that the image of $f$ is an open set $B\subseteq\C^*$. Suppose that $f$ is a locally trivial fibration over $B$ with fiber $F$. The monodromy of $f$ makes $H_i(F,\Q)$ into a $\pi_1(B)$-module. Let $K = \ker (\pi_1(B)\to \pi_1(\C^*))$. Then,
\[
\Tors_A H_i(X,\sL_X) \cong \frac{H_i(F,\Q)}{\langle \gamma \cdot a - a\mid \gamma\in K,a\in H_i(F,\Q)\rangle}.
\]
\end{enumerate} 
\end{cor}

%In relation to semi-simplicity, we provide the following generalization of \cite[Theorem 8.0.1]{EGHMW}.

In relation to semi-simplicity, we generalize \cite[Theorem 8.0.1]{EGHMW} as follows.

\begin{prop}\label{ss}
If $X$ is smooth and $f\colon X \to T$ is proper, then the Artinian modules $\Art H^i(X, \sL_X)$ are semi-simple $A$-modules.
\end{prop}

%%%%%%%%%%%%%%%%%%%%%%%%%%%

\subsection{Sketch of our construction of a MHS on $\Art H^i(X, \sL_X)$}
For the reader's convenience, we include here a brief description of our construction of the MHS on $\Art H^i(X, \sL_X)$. First, by using the projection formula, we have a natural isomorphism (see Proposition \ref{prop_iso1})
\begin{equation}\label{eq0}
H^i(X, \sL_X)\cong H^i\big(T, Rf_*\underline{\Q}_X \otimes_\Q \sL_X\big),
\end{equation}
where $Rf_*\underline{\Q}_X$ is regarded here as a complex of mixed Hodge modules on $T$. In Section \ref{sec:smooth}, we show that for a complex of mixed Hodge modules $\sM^\bullet$ on the affine torus $T$, there exists a unique maximal smooth mixed Hodge submodule of each $H^i(\sM^\bullet)$, which we denote by $\sM^i_s$.  By using the $t$-exactness of Gabber-Loeser's Mellin transformation \cite{GL96}, we next prove the following result (see Corollary \ref{cor_GL}).
\begin{theorem}\label{thm_intro}
For any integer $i$, there is a canonical isomorphism 
\begin{equation}\label{i1}
\Art H^i\big(T, \rat(\sM^\bullet)\otimes_\Q \sL_T\big)\cong H^0\big(T, \rat(\sM^i_s)\otimes_\Q \sL_T\big),
\end{equation}
where $\rat$ associates to a complex of mixed Hodge modules its underlying rational constructible complex.
\end{theorem}

In view of \eqref{eq0}, Theorem \ref{thm_intro} induces an isomorphism
\begin{equation}\label{eq_s0}
\Art H^i(X, \sL_X)\cong H^0\big(T, \rat(\sM^i_s)\otimes_\Q \sL_T\big)
\end{equation}
where $\sM^\bullet=Rf_*\underline{\Q}_X$.% is the pushforward mixed Hodge  complex on $T$. 

By the equivalence 
\[
\mathrm{MHM}(T)_s\cong \mathrm{VMHS}(T)_{ad},
\]
between the category $\mathrm{MHM}(T)_s$ of smooth mixed Hodge modules on $T$ and the category $\mathrm{VMHS}(T)_{ad}$ of admissible variations of mixed Hodge structures on $T$ (with quasi-unipotent monodromy at infinity),
there exists a quasi-unipotent admissible VMHS $\sV$ on $T$ such that, after a shift, the underlying local system of $\sM^i_s$ is isomorphic to the underlying local system $L$ of $\sV$, and for any point $x\in T$ the two mixed Hodge structures $(\sM^i_s)|_x$ and $\sV|_x$ coincide.

In theory, there are different methods to extract a MHS from a quasi-unipotent VMHS on $T$, for example, taking the central fiber of the Deligne extension. Nevertheless, we choose to simply take the stalk at a base point (on the universal cover $\wt T$ of $T$). More precisely, in Lemma \ref{lemma_ls}, we prove that
\begin{equation}\label{eq_lb}
H^0\big(T, \rat(\sM^i_s)\otimes_\Q \sL_T\big)\cong {L}_{\tilde{b}}
\end{equation}
where ${L}_{\tilde{b}}\coloneqq \pi^*(L)|_{\tilde{b}}$ and the $A$-module structure on ${L}_{\tilde{b}}$ is given by the inverse monodromy representation of $L$. Since as a local system on $T$, $L$ supports the VMHS $\sV$, the pullback $\pi^*(L)$ supports the VMHS $\pi^*(\sV)$, and there is a natural MHS on ${L}_{\tilde{b}}$. We define the MHS on $\Art H^i(X, \sL_X)$ to be the one of ${L}_{\tilde{b}}$ via \eqref{eq_s0} and \eqref{eq_lb}. More details are given in Section~\ref{sec_proof}. 

At this end, let us note that one can associate (generalized) cohomological Alexander modules $H^i(X,\sF \otimes_\Q \sL_X)$ to any $\Q$-constructible coefficients $\sF \in D^b_c(X)$ (with Definition \ref{def1} corresponding to the case of the constant sheaf $\underline{\Q}_X$). Moreover, if $\sF$ underlies a complex of mixed Hodge modules, then we get the following generalization of Theorem \ref{thm1}.
\begin{theorem}\label{thm2}
Let  $\sM^\bullet$ be a complex of mixed Hodge modules on $X$ and let $\sF=\rat(\sM^\bullet)$ be the underlying rational constructible complex. Then $\Art H^i(X, \sF\otimes_\Q \sL_X)$ has a natural mixed Hodge structure, satisfying properties analogous to those listed in Theorem \ref{thm1}. 
\end{theorem}
\begin{remark}
The proofs of the analogous properties (1-5) in Theorem \ref{thm2} are essentially the same as the ones in Theorem \ref{thm1}. To emphasize the geometric aspects of the paper, we will only prove the properties in Theorem \ref{thm1} and provide the construction of the MHS on $\Art H^i(X, \sF\otimes \sL_X)$ in Theorem \ref{thm2}. 
\end{remark}

%%%%%%%%%%%%%%%%%%%%%%%%%%%%%

\subsection{Comparison with earlier work}

The goal of this section is to discuss the differences between the present paper and  our earlier work \cite{EGHMW}, both in the methods used and in the scope of the results obtained.

As mentioned in Section~\ref{sec:setup}, this paper deals with Hodge-thoretic properties of the maximal Artinian submodules $\Art H^i(X,\sL_X)$ of the cohomological Alexander modules constructed from the pair $(X,f)$, where $X$ is a complex algebraic variety and $f:X\rightarrow T$ is an algebraic map to the complex affine torus $T$. We should start by comparing the setup with that of \cite{EGHMW}. In {\it loc.cit.}, a canonical and functorial mixed Hodge structure is constructed on $\Art H^i(X,\sL_X)$, where $X$ is required to be smooth, $n=1$ (i.e., $T=\C^*$), and $f$ induces an epimorphism on  fundamental groups.

Despite the fact that the present paper studies Alexander modules defined in more generality, we want to emphasize that it is not a generalization of \cite{EGHMW}, as the different methods used in the construction of the mixed Hodge structures in each paper result in these constructions being well suited to study different kinds of problems. In fact, there is very little overlap between the two papers, only Proposition~\ref{ss} and Corollary~\ref{van}(c) are generalizations of results in \cite{EGHMW}. The statement of Corollary~\ref{van}(c)  does not involve mixed Hodge structures, but the proof uses properties that the mixed Hodge structures in both papers have in common. The proofs of \cite[Theorem 8.0.1]{EGHMW} and its present generalization in Proposition~\ref{ss} both use the decomposition theorem of \cite{BBD}, and neither one of them uses mixed Hodge structures.

An essential difference between the two papers lies in the methods used in the construction of the mixed Hodge structure. The present paper uses algebraic methods: Gabber-Loeser's $t$-exactness of the Mellin transformation to translate the objects of study into certain hypercohomology groups of the torus $T$, Hain and Zucker's work on unipotent variations of mixed Hodge structures on $T$, and Saito's theory of Mixed Hodge modules. This construction yields a mixed Hodge structure on $\Art H^*(X,\sL_X)$ which is dependent on a base point on the universal cover of $T$, although different choices of base points yield non-canonically isomorphic mixed Hodge structures in general. On the other hand, the methods used in \cite{EGHMW} are analytic: we use a refinement (called ``thickening'') of Deligne's theory of mixed Hodge complexes of sheaves, which involves de Rham complexes. The resulting mixed Hodge structure in \cite{EGHMW} is independent of base points.

Another notable difference in the construction of the mixed Hodge structure in both papers is the following. The construction in \cite{EGHMW} uses that, in the one variable case, the torsion part of the Alexander modules is known to be quasi-unipotent, which follows from the structure theorem for cohomology jump loci (\cite{BLW,BW15,BW20}). On the other hand, the construction in the present paper does not use the quasi-unipotency of $\Art H^i(X,\sL_X)$. Instead, we obtain the quasi-unipotency of $\Art H^i(X,\sL_X)$ (which was not known in this generality) as a result of the construction of the MHS, using Saito's theory of mixed Hodge modules.

%\textcolor{magenta}{FIX} The first difference in scope is the mixed Hodge structure on homological Alexander modules. 
%The second main difference in scope is the kind of choices involved in the construction of a mixed Hodge structure. Both mixed Hodge structures are not preserved by the monodromy action in general, and instead they satisfy Theorem~\ref{thm1}(4). This is perhaps expected, as the limit mixed Hodge structure has the same property. Therefore, the construction must break the symmetry provided by the $\Z$-action. In \cite{EGHMW}, the choice is made by fixing the exponential function as \emph{the} infinite cyclic cover of $\C^*$ (the same choice made in the construction of the limit MHS in \cite[Section 11.2]{peters2008mixed}). In this paper, we instead choose a base point $b$ on the universal cover of $T$.

Because of these differences, the present paper and \cite{EGHMW} have different focuses. The focus of the latter is on the Hodge theory of the torsion part of the one variable Alexander modules, and its relation to other mixed Hodge structures through maps coming from geometry.% The methods used in that paper make the resulting mixed Hodge structure well-suited for comparison to other mixed Hodge structures.
\ For example, the diagram
\begin{center}
\begin{tikzcd}[row sep = 1.2em]
\ & \widetilde X \arrow[d,"p"]\\
F\arrow[ur,"i_\infty"]\arrow[r,"i"] & X
\end{tikzcd}
\end{center}
obtained from a lift of the inclusion $i$ of a generic fiber $F$ of $f$ to the infinite cyclic cover $\widetilde{X}$ induces canonical maps of rational vector spaces as follows:
$$
H^j(X,\Q)\xrightarrow{H^j(p)}\Art H^{j+1}(X,\sL_X)\xhookrightarrow{H^j(i_\infty)}H^j(F,\Q).
$$
%\textcolor{magenta}{is it worth it to say that the Alexander module should have the conjugate module structure or not?}
In \cite[Theorem 6.0.1, Corollary 7.2.3]{EGHMW}, we prove that $H^j(p)$ is a morphism of mixed Hodge structures, and that $H^j(i_\infty)$ is a morphism of mixed Hodge structures for every lift $i_\infty$ of any generic fiber $F$ if and only if $\Art H^{j+1}(X,\sL_X)$ is a semisimple $A\cong\Q[t^{\pm 1}]$-module. Among other things, we also obtain relationships with cup products and with the limit mixed Hodge structure (if $f$ is proper), which also involve maps coming from geometry \cite[Proposition 7.1.1, Theorem 9.0.7]{EGHMW}.

%The construction of the mixed Hodge structure described in this paper through algebraic methods is much shorter than the one in \cite{EGHMW}. 
We would like to highlight an important aspect of the work in this paper which is independent of mixed Hodge structures: the $t$-exactness of the Mellin transform gives us the isomorphism of $A$-modules in equation (\ref{eq_s0}), which provides a new useful tool for computing $\Art H^*(X,\sL_X)$, as we can see in  Proposition~\ref{prop:removeFibers} and the example of Section~\ref{sec:non-ss}. % and %\textcolor{magenta}{ADD NEW THING ABOUT REMOVING FIBERS}.

However, this paper does not address generalizations of the main results of \cite{EGHMW}, which answer how the MHS on $\Art H^*(X,\sL_X)$ relates geometrically to other important mixed Hodge structures in the literature. If $n=1$, the map $\iota:\Art H^{j+1}(X,\sL_X)\rightarrow H^{j}(F_b,\Q)$ from Theorem~\ref{thm1}, part (1) is not defined in a way in which we can easily compare it to $H^j(i_\infty)$ from \cite{EGHMW}, so it is hard to determine whether the relationship between mixed Hodge structures given by $\iota$ has a geometric meaning.

The two very different approaches used to define the MHS in \cite{EGHMW} and in the present paper, respectively, make it difficult to provide a direct comparison. So we ask the following. 
\begin{que}
If $T=\C^*$, is there any relation between our construction of the MHS on $\Art H^i(X, \sL_X)={\rm Tors}_A  H^i(X, \sL_X)$ and the one given in \cite{EGHMW}?
\end{que}
We expect that the two MHS are non-canonically isomorphic, and that, when the monodromy action is semi-simple, such an isomorphism can be made canonical.

\subsection{Structure of the paper}
%The paper is structured as follows. 
In Section \ref{S:HZ}, we give a brief overview of Hain-Zucker's theory of unipotent VMHS on the complex affine torus $T$. In Section \ref{sec:MT}, we use the Mellin transformation of Gabber-Loeser \cite{GL96} to reduce the proof of Theorem \ref{thm2} to the case $i=0$, $X=T$, $f=id_T$ the identity map, and $\sF$ a perverse sheaf underlying a mixed Hodge module (see Theorem \ref{thm_MHM}). Section \ref{sec:smooth} reduces the problem further to the case when the perverse sheaf is replaced by its maximal smooth sub-object (which underlies, up to a shift, an admissible VMHS). Section \ref{sec_proof} completes the proof of Theorems \ref{thm1} and \ref{thm2}, and of Proposition \ref{ss}. We also justify here Corollaries \ref{cor_dominant} and \ref{iso}, and present a few simple examples.
Proposition~\ref{prop:removeFibers} and Corollary~\ref{cor:homologyOpen} are proved in Section~\ref{sec:remove}. Finally, in Section \ref{sec:non-ss}, we give an example of a {\it singular} complex algebraic variety with a non-semisimple cohomological Alexander module.

\medskip

We assume reader's familiarity with the basic derived calculus and perverse sheaves, see, e.g., \cite{Di} or \cite{M19} for a quick introduction to these topics.

\medskip

\textbf{Acknowledgements.} We thank Mircea Mus\-ta\-\c{t}\u{a} and Greg Pearlstein for useful conversations. E. Elduque is partially supported by an AMS-Simons Travel Grant. 
L. Maxim is partially supported by the Simons Foundation (Collaboration Grant \#567077) and by the Romanian Ministry of National Education (CNCS-UEFISCDI grant PN-III-P4-ID-PCE-2020-0029).  B. Wang  is  partially  supported a Sloan Fellowship and a WARF research
grant.

%%%%%%%%%%%%%%%%%%%%%%%%%%%%%%

\section{Unipotent variations of mixed Hodge structure on $T$}\label{S:HZ}
In this section, we give a brief overview of Hain-Zucker's theory of unipotent VMHS on the complex affine torus $T=(\C^*)^n$.%\green{ I think in this section we don't lose anything by saying $T=(\C^*)^n$. Doing it this way makes it so we don't have to worry about the sentence ``We fix the good compactification $T=(\C^*)^n\subset \P^n$.'' below.} 

The main result of \cite{HZ1} says that the admissible unipotent VMHS on a smooth quasi-projective variety $Y$ (with base point $y$) with unipotency $\leq r$ correspond to mixed Hodge representations of $\Z[\pi_1(Y, y)]/J^{r+1}$, where $J$ is the augmented ideal, that is, the ideal generated by $\sigma-1$ for all $\sigma\in \pi_1(Y, y)$. In \cite{HZ2}, given a mixed Hodge representation of $\Z[\pi_1(Y, y)]/J^{r+1}$, the corresponding unipotent VMHS is constructed explicitly in the case when $W_1H^1(Y, \Q)=0$. 

\begin{remark} The MHS on $\Z[\pi_1(Y, y)]/J^{r+1}$ is defined using iterated integrals (see \cite{Hain} and \cite{Chen}). An iterated integral on the product of two spaces has a K\"unneth type decomposition as the sum of products of iterated integrals on each space. Moreover, we can use coordinate-wise compactifications of the affine torus $(\C^*)^n$ to define the MHS on $\Z[\pi_1(T, b)]/J^{r+1}$. Thus, to understand the MHS on $\Z[\pi_1(T, b)]/J^{r+1}$, it suffices to understand the MHS on $\Z[\pi_1(\C^*, b')]/J^{r+1}$, where $b'$ is a chosen base point of $\C^*$. The weight filtration on $\C[\pi_1(\C^*, b')]/J^{r+1}$ coincides with the filtration defined by the powers of $J$ (\cite[Page 248]{Hain}), and the Hodge filtration can be obtained by computing the iterated integrals on the unit circle. The $(-1, -1)$ subspace of $\C\green{[}\pi_1(\C^*, b')\green{]}/J^{r+1}$ is generated by 
%\[\log \tau=1-(1-\tau)+\frac{1}{2}(1-\tau)^2-\cdots
%\]
\[
\log \tau=(\tau-1)-\frac{1}{2}(\tau-1)^2+\cdots
\]
where $\tau\in \pi_1(\C^*, b')$ is the generator with winding number 1. 
\end{remark}

\begin{remark} \label{remark_HZ}
An admissible unipotent VMHS $\sV$ on $(\C^*)^n$ has the following explicit form (see \cite[(4.11)]{HZ2}). We fix the good compactification $(\C^*)^n\subset \P^n$. The canonical extension $\overline{\sV}$ of $\sV$ is a trivial vector bundle on $\P^n$. Moreover, the weight and Hodge filtrations on $\overline{\sV}$ are both induced by the global sections. In other words, there exists a MHS $V$ with weight filtrations $W_\bullet$ and Hodge filtrations $F^\bullet$, such that $\overline{\sV}=V\otimes_\C \sO_{\P^n}$, the weight filtration is given by $W_l\overline{\sV}=W_l V\otimes_\C \sO_{\P^n}$ and the Hodge filtration is given by  $F^p\overline{\sV}=F^p V\otimes_\C \sO_{\P^n}$. The logarithmic connection on $\overline{\sV}$ defines a local system structure on $\overline{\sV}|_{(\C^*)^n}$ such that the weight filtration is locally constant, but the Hodge filtration is not locally constant in general. 
%Thus, the weight filtration on  $\overline{\sV}|_{(\C^*)^n}$ is defined over $\Q$.
%\green{What is ``the'' logarithmic connection in this remark? I don't understand HZ87b. My guess is that if we have monodromy operators $\rho_1,\ldots ,\rho_n$, the connection is the following (using the unique nilpotent logarithms)
%\[
%\nabla  = d+\sum_j \frac{\log \rho_j}{2\pi i}\otimes \frac{dz_j}{z_j},
%\]
%but then again... HZ87b is confusing. I'm not saying we should add this formula here, but maybe we should say ``there exists a connection'' if we don't want to describe it?
%I also don't understand the implication in the last sentence.
%}
\end{remark}

The following theorem follows from the fact that the monodromy action of an admissible unipotent VMHS defines a mixed Hodge representation (\cite[Theorem 1.6]{HZ1}, \cite[Theorem (2.2)]{HZ2}). 
\begin{theorem}\label{thm_HZ}\normalfont{[Hain-Zucker]}
Let $\sV$ be an admissible quasi-unipotent {VMHS} on $T$. Fixing a base point ${b}\in T$, then for any $\sigma\in \pi_1(T,b)$, the induced map
\[
\log(\sigma^N): \sV|_b\to \sV|_b(-1)
\]
is a morphism of MHS, where $N$ is any positive integer such that $\sigma^N$ is unipotent. 
\end{theorem}

%%%%%%%%%%%%%%%

\section{Mellin transformation}\label{sec:MT}

The following proposition allows us to reduce Theorem \ref{thm2} to the special case when $X=T$ and $f=\id_T$ is the identity map. 
\begin{prop}\label{prop_iso1}
Under the notations of Theorem \ref{thm2}, we have a natural isomorphism
\begin{equation}\label{eq5}
H^i(X, \sF\otimes_\Q \sL_X)\cong H^i\big(T, Rf_*\sF\otimes_\Q \sL_T\big). 
\end{equation}
\end{prop}
\begin{proof}
We have that $H^i(X, \sF\otimes_\Q \sL_X)\cong H^i(T,Rf_*(\sF\otimes_\Q \sL_X))$. 
We claim that there is a canonical isomorphism
\[
Rf_*\sF\otimes_\Q \sL_T \cong Rf_*(\sF\otimes_\Q \sL_X).
\]
Since $\sL_X=f^*\sL_T$, there is a natural projection morphism (\cite[Lemma 1.4.1]{Sc03})
\[
Rf_*\sF\otimes_\Q \sL_T \to Rf_*(\sF\otimes_\Q \sL_X).
\]
Since $\sL_T$ is locally constant on $T$, one can easily check that the above morphism induces stalk-wise isomorphisms. Thus, the projection morphism is an isomorphism, and the assertion in formula \eqref{eq5} follows. 
\end{proof}

\begin{remark}\label{remark_inverse}
Even though $A=\Q[\pi_1(T)]\cong \Q[t_1^{\pm 1}, \ldots, t_n^{\pm 1}]$ is a commutative ring, an $A$-module has an induced $\Q[\pi_1(X)]$-module structure via $f:X \to T$, over the possibly noncommutative ring $\Q[\pi_1(X)]$. For this reason, we shall distinguish the right from the left $A$-modules. Following \cite{LMW20}, we let $D^b(A)$ be the bounded derived category of {\it right} $A$-modules. On the other hand, a $\pi_1(T)$-representation $V$ is naturally endowed with a left $A$-module structure. In this paper, we also regard such $V$ as a right $A$-module by using the conjugate $A$-module structure on $V$ given by: $v \cdot r= {\bar r} \cdot v$, for $v \in V$ and $r \in A$. (Here ${\bar \cdot}$ denotes the natural involution of $A$, sending each $t_i$ to ${\bar t_i} \coloneqq t_i^{-1}$.) This amounts to regarding the corresponding left $\Q[\pi_1(X)]$-module $V$ as a right $\Q[\pi_1(X)]$-module by setting $v \cdot \gamma \coloneqq  \gamma^{-1} \cdot v$, for all $v \in V$ and $\gamma \in \pi_1(X)$, and extending by linearity. 
\end{remark}

\begin{defn}[\cite{GL96}] The {\it Mellin transformation} is defined as
\begin{equation}\label{eq_mellin}
\Mel: D^b_c(T, \Q)\to D^b_{coh}(A), \quad \Mel(\sF)\coloneqq Rq_*(\sF\otimes_\Q \sL_T)
\end{equation}
where $q: T\to \mathrm{pt}$ is the projection to a point, and $D^b_{coh}(A)$ denotes the bounded coherent complexes of (right) $A$-modules.
\end{defn}

%We have the following.
\begin{lemma}\label{lemma_ls}
If $L$ is a $\Q$-local system on $T$ and $\tilde{b}$ is a fixed base point in the universal cover $\wt T$ of $T$, then there is a canonical isomorphism of $A$-modules \begin{equation}\label{eq33} \Mel(L)\cong {L}_{\tilde{b}}[-n],\end{equation} where ${L}_{\tilde{b}}\coloneqq \pi^*(L)|_{\tilde{b}}$ and the $A$-module structure on ${L}_{\tilde{b}}$ is given by the inverse monodromy representation of $L$. 
\end{lemma}
\begin{proof}
%{\color{red} Formula \eqref{eq33} can be deduced from \cite[Example 2.4]{LMW20}, by using well-known duality properties of the Mellin transformation (cf. \cite{GL96}). However, we give here a more direct proof.}
In fact, fix an isomorphism  $T=(\C^*)^n$. %\green{(Fix an isomorphism $T\cong (\C^*)^n$? I don't know what is meant here by ``splitting'')}\green{, let us fix a way to identify $\wt T=\C^n$, so we can say $\wt b\in \C^n$. Or we could replace all $(S^1)^n$'s by some $K$, a real torus, and $\wt K$ its universal cover.}. 
Let $(S^1)^n\subset (\C^*)^n$ be the compact subgroup of $T=(\C^*)^n$, and let $\pi_{S^1}: \R^n\to (S^1)^n$ be {the restriction of $\pi$}. Notice that $\sL_T|_{(S^1)^n}\cong R(\pi_{S^1})_!\underline{\Q}_{\R^n}$, where the $A$-module structure on $R(\pi_{S^1})_!\underline{\Q}_{\R^n}$ is induced by deck transformations. Then,
%\begin{align*}
%\Mel(L)&=Rq_*(L\otimes_\Q \sL_T)\\
%&\cong R\big(q|_{(S^1)^n}\big)_*\big(L|_{(S^1)^n}\otimes_\Q \sL_T|_{(S^1)^n}\big)\\
%&\cong R\big(q|_{(S^1)^n}\big)_*\big(L|_{(S^1)^n}\otimes_\Q R(\pi_{S^1})_!\underline{\Q}_{\R^n}\big)\\
%&\cong R\big(q|_{(S^1)^n}\big)_*\big(R(\pi_{S^1})_!\big((\pi_{S^1})^*(L|_{(S^1)^n})\otimes_\Q \underline{\Q}_{\R^n}\big)\big)\\
%&\cong R\big(q|_{(S^1)^n}\circ (\pi_{S^1})\big)_!(\pi_{S^1})^*\big(L|_{(S^1)^n}\big)\\  
%&\cong R\Gamma_c \big(\R^n, (\pi_{S^1})^*(L|_{(S^1)^n})\big)\\
%&\cong H_0 \big(\R^n, (\pi_{S^1})^*(L|_{(S^1)^n})\big)[-n]\\
%&\cong H_0\big(\C^n, \pi^*(L)\big)[-n]\\
%&\cong {L}_{\tilde{b}}[-n]
%\end{align*}
%\green{
%Might it make it more readable to write the same computation in terms of cohomology groups?}
\begin{align*}
H^i(\Mel(L))&=H^i(T,L\otimes_\Q \sL_T)\cong H^i\left((S^1)^n,L|_{(S^1)^n}\otimes_\Q \sL_T|_{(S^1)^n}\right) & \text{Homotopy eq.}\\
&\cong H^i\left((S^1)^n,L|_{(S^1)^n}\otimes_\Q (\pi_{S^1})_!\ul \Q_{\R^n}\right) &
\\
&\cong H^i\left((S^1)^n,(\pi_{S^1})_!(\pi_{S^1}^*(L|_{(S^1)^n})\otimes_\Q \ul \Q_{\R^n})\right)
&
\text{Projection formula}
\\
&\cong H^i_c\left((S^1)^n,(\pi_{S^1})_!\pi_{S^1}^*(L|_{(S^1)^n})\right)\cong H^i_c\left(\R^n,\pi_{S^1}^*(L|_{(S^1)^n})\right)
&
\text{$(S^1)^n$ is compact}
\\  
&\cong H_{n-i} \left(\R^n, \pi_{S^1}^*(L|_{(S^1)^n})\right)
&
\text{Poincar\'e duality}
\\
&\cong H_{n-i} \left(\C^n, \pi^*L\right)
&
\text{Homotopy eq.}
\\
&\cong {L}_{\tilde{b}} \text{ when }  i=n, \text{ and } 0 \text{ otherwise}.
%\left\{
%\begin{array}{ll}
%{L}_{\tilde{b}} & i=n\\
%0 & i\neq n
%\end{array}
%\right.
\end{align*}
%\green{Remove this, depending on what the formula looks like:(((} where the third isomorphism follows from the projection formula, the fourth  follows from the fact that $q|_{(S^1)^n}$ is proper, and the sixth follows from Poincar\'e duality.\green{)))} 
Notice that the $\pi_1(T)$-action on $H^0(\C^n, \pi^*(L))$ via deck transformations and the $\pi_1(T)$-action on ${L}_{\tilde{b}}$ via monodromy  are inverse to each other. Remark \ref{remark_inverse} provides a conceptual explanation for this fact. 
\end{proof}

%\begin{remark}
%Formula \eqref{eq33} can be also deduced from \cite[Example 2.4]{LMW20}, by using well-known duality properties of the Mellin transformation (cf. \cite{GL96}). \textcolor{blue}{should we remove this remark?}
%\end{remark}

\begin{prop}\label{prop_keyiso}
Let $L$ be a $\Q$-local system on $T$, and let $V$ be the $A$-module associated to the monodromy representation of $L$. For any $\Q$-constructible complex $\sF$, there is a natural isomorphism of $\Q$-vector spaces
\begin{equation}\label{eq34}
H^i\big(T, \sF\otimes_\Q L\big)\cong H^i\big(\Mel(\sF)\otimes^L_A V\big) 
\end{equation}
where $\otimes^L_A$ denotes the derived tensor product of right and, resp.,  left $A$-modules.
\end{prop}
\begin{proof}
By the projection formula, we have:
$$Rq_*(\sF\otimes_\Q \sL_T)\otimes^L_A V\cong Rq_*(\sF\otimes_\Q \sL_T \otimes^L_A q^*V) \cong Rq_*(\sF \otimes_\Q L).$$
The assertion follows by taking cohomology on both sides.
\end{proof}

Let us also recall here the following important result from \cite[Theorem 3.4.1]{GL96} (see also \cite[Theorem 3.2]{LMW18}):
\begin{theorem}[Gabber-Loeser]\label{thm_GL}
The Mellin transformation $\Mel: D^b_c(T, \Q)\to D^b_{coh}(A)$ is a t-exact functor with respect to the perverse t-structure on $D^b_c(T, \Q)$ and the standard t-structure on $D^b_{coh}(A)$. 
\end{theorem}

By Theorem \ref{thm_GL}, for any $\sF \in D^b_c(T, \Q)$, we have natural isomorphisms
$$H^i(\Mel(\sF)) \cong \Mel(^p\sH^i(\sF))\cong H^0(\Mel(^p\sH^i(\sF))),$$
where $^p\sH^i(-)$ denotes the perverse cohomology functor.
This yields the following.
\begin{cor}\label{cor_GL}
Let $\sF$ be a $\Q$-constructible complex on $T$. Then
\[
H^i(T, \sF\otimes_\Q \sL_T)\cong H^0\big(T, \,^p\sH^i(\sF)\otimes_\Q \sL_T\big). 
\]
In particular, if $\sP$ is a $\Q$-perverse sheaf on $T$, then for any $i\neq 0$,
\[
H^i(T, \sP\otimes_\Q \sL_T)=0. 
\]
\end{cor}

If $\sM^\bullet$ is a complex of mixed Hodge modules and $\sF\cong \rat(\sM^\bullet)$ is the underlying $\Q$-constructible complex, then $^p\sH^i(\sF)\cong \rat(H^i(\sM^\bullet))$. Therefore, Theorem \ref{thm2} reduces to proving the following result. 
\begin{theorem}\label{thm_MHM}
Let $\sM$ be a mixed Hodge module on $T$ and let $\sP=\rat(\sM)$ be the underlying perverse sheaf. For a choice of $\wt b\in \wt T$, the submodule $\Art H^{0}(T, \sP\otimes_\Q \sL_T)$ has a natural mixed Hodge structure. \end{theorem}

%%%%%%%%%%%%%%%%%%

\section{The maximal smooth sub-objects}\label{sec:smooth}
Given a perverse sheaf $\sP$ on a pure-dimensional complex manifold $Y$, a sub-object in the abelian category of perverse sheaves
\[
\sP' \hookrightarrow \sP
\]
is called {\it smooth} if $\sP'$ is the shift of a local system on $Y$. Among all smooth sub-objects, there exists a unique maximal one, which we call the {\it maximal smooth sub-object} of $\sP$, and we denote it by $\sP_{s}$. Consider the short exact sequence of perverse sheaves on $Y$,
\[
0\to \sP_{s}\to \sP\to \sP/\sP_{s}\to 0.
\]
Since extensions of smooth objects in the category of perverse sheaves are smooth, the quotient $\sP/\sP_{s}$ does not contain any nontrivial smooth sub-object. 

A sub-object $\sP'$ of $\sP$ is called {\it constant} if it is the shift of a global constant local system on $Y$. Among all constant sub-objects, there exists a unique maximal one, which we call the {\it maximal constant sub-object} of $\sP$, and we denote it by $\sP_c$. 

The maximal constant sub-object can be characterized by the following result. 
\begin{lemma}\label{lemma_composition} Let $q: Y\to \mathrm{pt}$ be the projection to a point. 
The sub-object $\sP_c$ is equal to the image of the composition
\begin{equation}\label{eq_composition}
q^*\big(\tau^{\leq -\dim Y}Rq_{*}\sP\big)\to q^*Rq_{*}\sP\to \sP,
\end{equation}
where the first morphism is induced by the truncation morphism $\tau^{\leq -\dim Y}Rq_{*}\sP \to Rq_{*}\sP$ and the second is the adjunction morphism. \end{lemma}
\begin{proof}
Since $\sP$ is a perverse sheaf on $Y$, we have $H^i(Rq_*\sP)=0$ when $i<-\dim Y$. Thus, $\tau^{\leq -\dim Y}Rq_{*}\sP$ has cohomology only possibly in degree $-\dim Y$. Hence, $q^*\big(\tau^{\leq -\dim Y}Rq_{*}\sP\big)$ is the shift of a global constant local system on $Y$. As a quotient, the image is also a constant sub-object of $\sP$. 

On the other hand, since the morphisms in \eqref{eq_composition} come from {natural transformations}, we have a commutative diagram
\[
\begin{tikzcd}[row sep = 1.2em]
q^*\big(\tau^{\leq -\dim Y}Rq_{*}(\sP_c)\big)\arrow[r,"\sim"]\arrow[d]& q^*Rq_{*}(\sP_c)\arrow[r,"\sim"]\arrow[d]& \sP_c\arrow[d,hookrightarrow]\\
q^*\big(\tau^{\leq -\dim Y}Rq_{*}\sP\big)\arrow[r]& q^*Rq_{*}\sP\arrow[r]& \sP.
\end{tikzcd}
\]
Since $\sP_c$ is the shift of a global constant local system, one can easily check that the composition of the first row is an isomorphism. Therefore, the image of $q^*\big(\tau^{\leq -\dim Y}Rq_{*}(\sP_c)\big)$ in $\sP$ is equal to $\sP_c$. Since the composition factors through $q^*\big(\tau^{\leq -\dim Y}Rq_{*}\sP\big)$, we know that the image of $q^*\big(\tau^{\leq -\dim Y}Rq_{*}\sP\big)$ in $\sP$ contains $\sP_c$. 
\end{proof}

By Theorem \ref{thm_GL}, the Mellin transformation $\Mel$ restricts to a functor
\[
\Mel: \Perv(T, \Q)\to A\textrm{-}\mathrm{Mod}_{coh}
\]
from the abelian category of $\Q$-perverse sheaves to the abelian category of finitely generated $A$-modules. 

The following proposition is comparable to \cite[Lemma 5.1]{LMW20}. 

\begin{prop}\label{prop_s0}
Let $\sP$ be a perverse sheaf on $T$, with maximal smooth sub-object $\sP_s$. Then
\begin{equation}\label{mels}
\Mel(\sP_s)=\Art \Mel(\sP)
\end{equation}
as submodules of $\Mel(\sP)$. 
\end{prop}
\begin{proof}
%{\color{red} The only concern I have about this proof is that, while shorter, it relies on big guns. This may prompt referees to say that our results are immediate consequence of work by others. Otherwise, I'm fine with it.}
%\green{
%[[[This version of the same proof shows Corollary~4.3 directly, without showing Proposition~4.2 first (from the preivous version), so it saves some space.]]]} {\color{red} My personal preference goes to the old version, it was more transparent, not just algebra magic.}\green{ Feel free to change it.}
Let $N, M$ be $A$-modules such that $N$ is Artinian, and let us denote $N^\vee = \Hom_{\Q}(N,\Q)$. By the Local Duality theorem (see \cite[Lecture 11]{I+07} or \cite[5.1, 5.2]{smith}),
\[
R\Hom_A(N,M)
\cong
M\otimes^L_A R\Hom_A(N,A)
\cong
M\otimes^L_A \Ext^n_A(N,A) [-n]
\cong
M\otimes^L_A N^\vee [-n]
\]
Let $L_N$ be a $\Q$-local system and let $N$ be its stalk, seen as an Artinian $A$-module. Note that by Lemma~\ref{lemma_ls}, $N\cong \Mel(L_N[n])$. Also, to the dual local system $L^\vee$ corresponds the module $N^\vee = \Hom_{\Q}(N,\Q)$. Let us apply $H^0$ to the equation above, with $M=\Mel(\sP)$:
\begin{align*}
\Hom_A(N,\Mel(\sP)) &\cong
H^{-n}(\Mel(\sP)\otimes_A^L N^\vee) \\
&\hspace{-0.8mm}\overset{\Mel}{\cong}
\hspace{-0.5mm}H^{-n}(T,\sP\otimes L^\vee)
&
\text{by Proposition~\ref{prop_keyiso}}\\
&\cong
\Hom_{\Perv(T)}(L[n],\sP).
\end{align*}
%This tells us that applying $\Mel$ induces a bijection between smooth subobjects of $\sP$ and Artinian submodules of $\Mel(\sP)$. In particular, the maximal smooth subobject corresponds to the maximal Artinian submodule of $\Mel(\sP)$.

Since the above isomorphisms are functorial in $L_N$, applying them to the inclusion map  $\Art\Mel(\sP)\to \Mel(\sP)$, we obtain a map $L_N[n]\to \sP$, whose image is clearly contained in $\sP_s$. We conclude that $\Art \Mel(\sP)\subseteq \Mel(\sP_s)$. The reverse inclusion is clear.
\end{proof}

\begin{prop}\label{prop_subMHM}
Let $\sM$ be a mixed Hodge module on the affine torus $T$. Let $\sP=\rat(\sM)$ be the underlying perverse sheaf. Then there exists a {unique} sub-mixed Hodge module $\sM_s$ of $\sM$ such that $\rat(\sM_s)=\sP_s$. 
\end{prop}
Since the functors of pullback, pushforward and truncations over a point lift to complexes of mixed Hodge modules, Lemma \ref{lemma_composition} implies the following lemma. 
\begin{lemma}\label{lemma_constant}
Let $\sM$ be a mixed Hodge module on a complex algebraic variety $Y$. Let $\sP=\rat(\sM)$. Then there exists a sub-mixed Hodge module $\sM_c$ of $\sM$ such that $\rat(\sM_c)=\sP_c$. 
\end{lemma}

\begin{proof}[Proof of Proposition \ref{prop_subMHM}]
Since $\sP=\rat(\sM)$, on a Zariski open subset of $T$, $\sP$ is the shift of a local system which has quasi-unipotent monodromy around the boundaries. Since $\sP_s$ is smooth and is a sub-perverse sheaf of $\sP$, $\sP_s$ is a quasi-unipotent local system on $T$. In other words, there exists a finite covering map $g: T'\to T$ such that $g^*(\sP_s)$ is the shift of a unipotent local system. 

{If $g^*(\sP_s)$ is non-zero, then $(g^*(\sP_s))_c$ is not trivial, since $\pi_1(T')$ is abelian. }Iterating Lemma \ref{lemma_constant}, we know that there exists a sub-mixed Hodge module $\sM_1$ of $g^*(\sM)$ such that $\rat(\sM_1)=\rat(g^*\sM)_s$. Here, we note that $g^*(\sP_s)=(g^*\sP)_s$. Define $\sM_s$ to be the image of the composition
\[
Rg_*\big({\sM_1}\big)\to Rg_*\big(g^*\sM\big)\to \sM,
\]
where the second map is induced by the adjunction morphism  $Rg_!g^! \to id$ together with the fact that, since $g$ a finite covering map, the functors $Rg_!=Rg_*$ and $g^!=g^*$ preserve perverse sheaves (see, e.g., \cite[Corollary 5.2.15]{Di}).
Then $\rat(\sM_s)=\sP_s$. {Uniqueness follows from the fact that $\rat$ is faithful and exact, so the set of subobjects of $\sM$ injects into the set of subobjects of $\rat(\sM)$.}
\end{proof}

%%%%%%%%%%%%%%%%%%%%%%%%

\section{Proof of the main theorems}\label{sec_proof}

We can now complete the proof of Theorem \ref{thm_MHM}, and hence of Theorem \ref{thm2}.
\begin{proof}[Proof of Theorem \ref{thm_MHM}]
Let $\sM$ be a mixed Hodge module on $T$, let $\sP=\rat(\sM)$ and let $b = \pi(\tilde b)$. Then we have canonical isomorphisms
\begin{equation}\label{eq_5}
S_0H^0(T, \sP\otimes \sL_T)  \cong S_0\Mel(\sP)\overset{\eqref{mels}}{=} \Mel(\sP_s) \overset{\eqref{eq33}}{\cong}  L_{\tilde{b}}\cong L_{b},
\end{equation}
where $L=\sP_s[-n]$ is the underlying local system of $\sP_s$ and $L_{\tilde{b}}=\pi^*(L)|_{{\wt b}}$.
Since $L$ supports a VMHS {(Proposition~\ref{prop_subMHM})}, the $\Q$-vector space {$L_{b}$}, and hence $S_0H^0(T, \sP\otimes_\Q \sL_T)$, carry natural mixed Hodge structures. 
\end{proof}

\begin{proof}[Proof of property (1) of Theorem \ref{thm1}]

%\green{[[[I believe is the same proof as in the previous version, but in a different order, which made it shorter, and, for me, easier to follow]]]}

We first define the map $\iota$. By Proposition \ref{prop_iso1} and Corollary \ref{cor_GL}, we have the isomorphisms
\begin{equation}\label{eqa}
H^i(X, \sL_X)\cong H^i\big(T, Rf_*\underline{\Q}_X\otimes \sL_T\big)\cong H^0\big(T, \,^p\sH^i\big(Rf_*\underline{\Q}_X\big)\otimes \sL_T\big).
\end{equation}
Let $\sP=\,^p\sH^i\big(Rf_*\underline{\Q}_X\big)$ with maximal smooth sub-object $\sP_s$, and let $L = \sP_s[-n]$. 
Combining \eqref{eqa} and \eqref{eq_5}, we get the following isomorphisms: %{\color{red} (The above paragraph does not have MHS in it, so I'd rather not mention MHS here).}
%\textcolor{magenta}{This was written before and later erased, but I think it would be good for the reader to review here that (12) is a morphism of MHS by definition, and we define the MHS on $S_0H^i(X,\sL_X)$ so that (13) is a MHS iso. So (14) is an isomorphism of MHS}
\begin{equation}\label{eq:Alex-Stalk}
S_0H^i(X, \sL_X)\overset{\eqref{eqa}}\cong S_0H^0(T,\sP\otimes \sL_T)\overset{\eqref{eq_5}}\cong L_{b}.
\end{equation}
Let $U\subset T$ be a nonempty Zariski open set over which $f$ is a topologically locally trivial fibration (see, e.g., \cite[Corollary 5.1]{Ve76}). In particular, $Rf_*\underline{\Q}_X$ is a locally constant complex on $U$, and hence the restriction of $\sP$ to $U$ is also smooth (see \cite[Section 4]{LMW20} for the definition and properties of locally constant constructible complexes, and also \cite{B08} where such complexes are called {cohomologically locally constant}). Let us denote the resulting local system $L_U^\dagger\coloneqq \sP|_U[-n]$. Then we have $(R^{i-n}f_*\ul \Q_X)|_U \cong L_U^\dagger$. % (see e.g., \cite[Exercise 9.3.27]{M19}).}

The inclusion $\sP_s\subseteq \sP$ yields an injection $\wt \iota\colon L|_U\subset L^\dagger_U$. The map $\iota$ will be defined from the stalk of $\wt\iota$, as the following composition:
\begin{equation}\label{eq:defOfIota}
\iota\colon S_0H^i(X, \sL_X)\overset{\eqref{eq:Alex-Stalk}}{\cong} L_b ={(L|_U)}|_{b} \xrightarrow{\wt\iota|_b} {(L^\dagger_U)}|_{b} \cong  H^{i-n}(F_b, \Q),
\end{equation}
where the last isomorphism uses the fact that $f$ is a locally trivial fibration over $U$.
%The last isomorphism is due to the fact that the restriction of a fiber bundle over a contractible subset is a trivial bundle.

To show that $\iota$ is a MHS morphism, it will suffice to verify that all morphisms in \eqref{eq:defOfIota} are MHS morphisms. The isomorphism \eqref{eq:Alex-Stalk} is a MHS isomorphism because of the way we define the MHS on $S_0H^i(X, \sL_X)$. % {\color{red} by using \eqref{eqa} and \eqref{eq_5}}. 
For $\wt\iota|_b$, it suffices to show that $\wt\iota$ is a morphism of VMHS. For this, notice that the VMHS structure of both $L^\dagger_U$ and $L$ is induced from the mixed Hodge module structure on {$\sP$}.  Finally, the natural base change morphism  $(R^{i-n}f_*\ul \Q_X)_b \to H^{i-n}(F_b, \Q)$ is an isomorphism of MHS.
%$H^{i-n}(F_b, \Q)$ is just the stalk of the VMHS $L_U^\dagger$ on $U$.}\textcolor{magenta}{Our point is that perhaps it isn't clear why this stalk of this VMHS coincides with Deligne's MHS on the fiber. At least it is not immediately clear to us.}
%\green{No idea why ${(L^\dagger)}|_{b} \cong  H^{i-n}(F_b, \Q)$ as MHS. I assume that the base change isomorphism 
%\[
%\begin{tikzcd}[ampersand replacement = \&]
%F \arrow[r,hookrightarrow,"\wt i"]\arrow[d,"f|_F"']\arrow[phantom,dr,very near start,"\lrcorner"]
%\&
%X \arrow[d,"f"]
%\\
%b \arrow[r,"i"] \&
%T
%\end{tikzcd}
%\]
%$i^*\circ R(f|_F)_*\cong Rf_*\circ \wt i^*$ is a MHM isomorphism.
%}
\end{proof}

Before proving the properties (2-5) in Theorem \ref{thm1}, we notice that the $\pi_1(T)$-action on 
$\Art H^i(X, \sL_X)$  is quasi-unipotent. This follows from Proposition \ref{prop_subMHM}, applied to the perverse sheaf $\sP=\,^p\sH^i\big(Rf_*\underline{\Q}_X\big)$. We can further reduce to the case when the $\pi_1(T)$-action on $\Art H^i(X, \sL_X)$ is unipotent, as follows. Let $h_T: T'\to T$ be any finite covering map. Then, we can choose a (universal) covering map $\pi'\colon\wt T\to T'$ {with the following commutative diagram}:
\[
\begin{tikzcd}[row sep = 1.5em]
\widetilde{X}\arrow[r,"p'",dashrightarrow]\arrow[d,equals]
&
X'\arrow[r,"f'"]\arrow[d,"h"]\arrow[dr,phantom,very near start]
&
T'\arrow[d,"h_T"] & 
\wt T \arrow[l,dashrightarrow,"\pi'"',"\exists"]\arrow[d,equals]
\\
\wt X \arrow[r,"p"]
&
X\arrow[r,"f"]
&
T
&
\wt T \arrow[l,"\pi"']
\end{tikzcd}
%\xymatrix{
%\widetilde{X}\ar[r]^{p'}\ar[dr]_{p}&X'\ar[r]^{f'}\ar[d]^{h}&T'\ar[d]^{h_T}\\
%&X\ar[r]^f&T
%}
\]
{where the middle square is a Cartesian product and $p'$ is the pullback of $\pi'$ by $f'$. Using the top row, we can define $\sL_{X'} = p'_!\underline \Q_{\wt X}$ and the corresponding Alexander modules.} By \eqref{eq_p} and the fact that $h$ is a proper map, we have {the following $\Q[\pi_1(T')]$-module isomorphisms}
\begin{multline}\label{eq_2line}
H^i(X', \sL_{X'})= H^i(X', p'_!\underline{\Q}_{\widetilde{X}}) \cong H^i(X, h_*p'_!\underline{\Q}_{\widetilde{X}})\\
\cong H^i(X, h_!p'_!\underline{\Q}_{\widetilde{X}})\cong H^i(X, p_!\underline{\Q}_{\widetilde{X}})= H^i(X, \sL_X).
\end{multline}
%where the second isomorphism follows from associativity of derived pushforward, and the third follows from $h$ being a proper map. %Note that the Alexander modules in \eqref{eq_2line} are modules over different rings, namely $\Q[\pi_1(T')]$ and $\Q[\pi_1(T)]$, respectively. However, these module structures are defined as deck actions, hence the inclusion $\pi_1(T')\subseteq \pi_1(T)$ makes \eqref{eq_2line} a $\Q[\pi_1(T')]$-module isomorphism.

\begin{lemma}\label{lem:finiteCover}
The isomorphism in (\ref{eq_2line}) induces a MHS isomorphism $\Art H^i(X', \sL_{X'})\cong \Art H^i(X, \sL_X)$, where both MHS are constructed using the same base point $\widetilde{b}\in\widetilde T$.
\end{lemma}
\begin{proof}
First, note that $\Q[\pi_1(T)]$ is a finitely generated $\Q[\pi_1(T')]$-module. {Thus, the definition of $\Art H^i(X, \sL_X)$ does not depend on whether  it is considered as a $\Q[\pi_1(T)]$-module or a $\Q[\pi_1(T')]$-module.}

Consider the isomorphisms in \eqref{eq:Alex-Stalk} applied both to $(X,f)$ and $(X',f')$. We use $\Mel_T$ and $\Mel_{T'}$ to denote the Mellin transform on the tori $T$ and $T'$ respectively. Let $L=\pH^i(Rf_*\underline \Q_{X})_s[-n]$ and $L' = \pH^i(Rf'_*\underline \Q_{X'})_s[-n]$.
The isomorphisms \eqref{eq:Alex-Stalk} form the rows of the following diagram. 
\[
\begin{tikzcd}[row sep = 1.2em]
\Art H^i(X',\sL_{X'}) \arrow[r,leftrightarrow] \arrow[d,leftrightarrow,"(\ref{eq_2line})"]&
\Mel_{T'}(L') \arrow[r,leftrightarrow] \arrow[d,leftrightarrow,"L'\cong h_T^*L"]&
L'_{\wt b}\arrow[r,leftrightarrow]\arrow[d,leftrightarrow,"L'\cong h_T^*L"]&
 L'_{\pi'(\tilde b)}\arrow[d,leftrightarrow,"L'\cong h_T^*L"]\\
\Art H^i(X,\sL_{X}) \arrow[r,leftrightarrow] &
\Mel_{T}(L) \arrow[r,leftrightarrow] &
L_{\wt b}\arrow[r,leftrightarrow]&
 L_{\pi(\tilde b)} %\node [anchor=east,overlay,inner sep=0, outer sep=0] at (m-2-2 -| 0.5\textwidth,0) {\qedhere};
\end{tikzcd} %\qedhere
\]
Since $h_T: T'\to T$ is a covering map, we get an isomorphism of VMHS, $L'\cong h_T^*L$, 
%\begin{equation}\label{pulll}
%L'\cong h_T^*L.
%\end{equation}
which together with the projection formula yields that
$$
\Mel_{T'}(L') {\cong} R\Gamma(T', h_T^*L \otimes \sL_{T'})  \cong 
R\Gamma(T, {h_T}_*(h_T^*L \otimes  \sL_{T'}))
\cong R\Gamma(T, L \otimes {h_T}_*\sL_{T'}) \cong \Mel_{T}(L).
$$
Here, the last isomorphism uses ${h_T}_*\sL_{T'} \cong \sL_T$, as in \eqref{eq_2line}.
By applying $\pi'^*$ and the stalk to the isomorphism $L'\cong h_T^*L$ we get the third vertical isomorphism  in the diagram. The last vertical isomorphism comes simply from taking stalks; this is a MHS isomorphism since  $L'\cong h_T^*L$ is a VMHS isomorphism. The assertion in the lemma then follows from the fact that the diagram commutes, which is a straightforward verification.
\end{proof}

\begin{proof}[Proof of Properties (2-5) of Theorem \ref{thm1}]
Property (3) follows from the naturality of our construction. More specifically, the isomorphism {(\ref{eq_5}) can be seen as a natural transformation in the variable $\mathcal P$, and we can apply it to the adjunction MHM morphism $\pH^i(Rg_*\ul \Q_Y)\to \pH^i(Rg_*R\phi_*\ul\Q_X)\cong \pH^i(Rf_*\ul \Q_X)$}.

Essentially, properties (2), (4) and (5) are consequences of the work of Hain-Zucker \cite{HZ1} and \cite{HZ2}. {Applying Lemma~\ref{lem:finiteCover} to a suitable finite cover of $T$, we can assume that the action of $\pi_1(T)$ on $\Art H^i(X, \sL_X)$ is unipotent, or equivalently, the VMHS $\pH^i(Rf_*\ul\Q_X)_{s}[-n]$ has unipotent monodromy.} %[[[This discussion was before this proof in the previous version]]]}

Property (2) and the second part of property (5) follow from the fact that for a unipotent admissible VMHS on $T$, the MHS on different points of $T$ are non-canonically isomorphic to each other. This is reviewed in Remark \ref{remark_HZ}. Property (4) follows immediately from Theorem \ref{thm_HZ}. Finally, we prove the first part of property (5). If the underlying local system of a VMHS $\sV$ on $T$ is semisimple, and hence trivial, then the associated mixed Hodge representation is equal to the trivial representation on a given MHS. Thus, both the VMHS $\sV$ and its pullback $\pi^*(\sV)$ are trivial families of MHS. Therefore, for any two choices $\tilde b,\tilde b'\in \wt T$, there is a natural isomorphism between the MHS $\pi^*(\sV)_{\tilde b}$ and $\pi^*(\sV)_{\tilde b'}$. 
\end{proof}

\begin{proof}[Proof of Proposition \ref{ss}]
By the decomposition theorem of \cite{BBD}, there is a decomposition
$$Rf_*\underline{\Q}_X \simeq \bigoplus_{\lambda \in \Lambda} \sP_\lambda[d_\lambda],$$
where $\Lambda$ is a finite index set, $d_\lambda \in \Z$, and each $\sP_\lambda$ is a simple $\Q$-perverse sheaf on $T$.

In view of our description of $\Art H^i(X, \sL_X)$, it suffices to show that if $\sP$ is a smooth $\Q$-perverse sheaf, then $H^0(T, \sP \otimes \sL_T)$ is a simple $A$-module. Since $\sP$ is smooth, there is a $\Q$-local system $L$ on $T$ so that $\sP \cong L[n]$. Moreover, since $\sP$ is simple, the local system $L$ is simple (that is, the corresponding $\pi_1(T,b)$-representation is simple). In particular, the stalk $L_b \cong L_{\tilde{b}}$ is a simple $A$-module. Finally, using Lemma \ref{lemma_ls}, we get
$$H^0(T, \sP \otimes \sL_T)=\Mel(\sP) \cong L_{\tilde{b}},$$
which concludes our proof.
\end{proof}

\begin{proof}[Proof of Corollary \ref{van}]
Statements (a), (b) and (d) follow immediately from Theorem \ref{thm1} (2). The weight filtration $H^{i-n}(F_b, \Q)$ is only nontrivial between degree $\max\{0, 2i-2n-2d\}$ and $\min\{2i-2n, 2d\}$. Moreover, when $F_b$ is smooth, then the weight filtration $H^{i-n}(F_b, \Q)$ is only nontrivial between degree $ i-n$ and $\min\{2i-2n, 2d\}$.  Thus, statement (c) follows from Theorem  \ref{thm1} (4). This is the same idea of proof as \cite[Corollary 7.4.2]{EGHMW}, which only uses the bound on the weights and the analogous statement to Theorem~\ref{thm1} (4).%fact that there mixed Hodge structure morphism into the $(-1)$st Tate twist, which decreases the weight by $2$.
\end{proof}

\begin{proof}[Proof of Corollary \ref{iso}]
If $f:X \to T$ is a topologically locally trivial fibration, the $\Q$-constructible complex $Rf_*(\underline{\Q}_X)$ is locally constant on $T$, and hence $R^kf_*(\underline{\Q}_X)$ are local systems on $T$, for all $k \in \Z$. Moreover, it follows from \cite[Section 4]{LMW20} that 
the perverse cohomology sheaves 
$^p\sH^i\big(Rf_*\underline{\Q}_X\big)$ ($i \in \Z$) are smooth, with
$$^p\sH^i\big(Rf_*\underline{\Q}_X\big)\cong (R^{i-n}f_*\underline{\Q}_X)[n].$$
The assertion follows now by tracing the construction of the map $\iota$ in Theorem \ref{thm1}(1).
\end{proof}

%We conclude this section with the following simple examples.

%\begin{example}
%Let $X=T \times \C^d$ and let $f:X \to T$ be the projection on $T=(\C^*)^n$. Corollary \ref{iso} yields immediately that $H^i(X, \sL_X)=0$ for all $i \neq n$, and  $H^n(X, \sL_X)\cong \Q$ (compare with \cite[Example 2.8]{DM07}).
%\end{example}

%\begin{example}
%If $X$ is smooth and $f:X \to T$ is proper, then the general fiber of $f$ is a complete smooth complex algebraic variety, whose rational cohomology groups have pure Hodge structures. In particular, Theorem \ref{thm1} yields in this case that the mixed Hodge structure on $\Art H^i(X, \sL_X)$, if not trivial, is pure of weight $i-n$.
%\end{example}

%%%%%%%%%%%%%%%%%%%%%%%%%%%%%

\section{Behavior of Alexander modules after removing fibers}\label{sec:remove}

In this section, we prove Proposition~\ref{prop:removeFibers} and Corollary~\ref{cor:homologyOpen}. 

Let $X$ be an algebraic variety, endowed with an algebraic map $f\colon X \to T$. Let $Z$ be a proper closed subset of $T$, with complement $U=T\setminus Z$, and let $Y=f^{-1}(U)$. Consider the following commutative diagram, where the horizontal arrows are open and closed embeddings, and the vertical arrows are restrictions of $f$:
\begin{equation}\label{eq:openClosed}
\begin{tikzcd}[row sep = 1.5em]
Y\arrow[r,hookrightarrow,"\tilde{\jmath}"]\arrow[d,"f_Y"] &
X\arrow[d,"f"] &
f^{-1}(Z) \arrow[l,hookrightarrow,"\tilde{\imath}"']\arrow[d,"f_Z"]
\\
U \arrow[r,hookrightarrow,"j"] &
T &
Z \arrow[l,hookrightarrow,"i"']
\end{tikzcd}
\end{equation}

%We will use the following lemma.
\begin{lemma}\label{lem:smoothPerverse}
Let $\sF$ be a $\Q$-local system on $T$ and $\sG$ be a perverse sheaf supported on a proper closed set $Z$ of $T$. Then:
\[
\Hom_{\Perv(T)}(\sF[n],\sG) =
\Hom_{\Perv(T)}(\sG,\sF[n]) =
0.
\]
\end{lemma}
\begin{proof}
First recall that perverse sheaves on a space of complex dimension $d$ are supported in cohomological degrees $[-d,0]$ (see, e.g., \cite[Exercise 8.3.5]{M19}). Let $i\colon Z=\operatorname{supp} \sG\to T$ be the closed inclusion. By the attaching triangle, we get $\sG\cong i_*i^*\sG$. Since $i^*\sG$ is perverse on $Z$ (see, e.g., \cite[Corollary 8.2.10]{M19}) and $\dim Z<n$, it follows that $i^*\sG$ is supported in cohomological degrees $[-n+1,0]$. By the exactness of $i_*$, the same is true for $\sG$. Since $\sF[n]$ is supported on cohomological degree $-n$, and since there are no nonzero morphisms from a complex in degrees at most $-n$ to a complex supported in degrees at least $-n+1$, we get that $\Hom_{\Perv(T)}(\sF[n],\sG) =0$. Then, by Verdier duality, $\Hom_{\Perv(T)}(\sG,\sF[n]) =0$.
\end{proof}

\begin{proof}[Proof of Proposition~\ref{prop:removeFibers}]
We use the notation in diagram (\ref{eq:openClosed}).
\begin{enumerate}[wide,
 labelwidth=!,
  %labelindent=0pt
  ]
\item By Proposition~\ref{prop_iso1} and Corollary~\ref{cor_GL}, for any $k$:
%\[
%H^i(X,\sL_X) \cong H^i(T,Rf_*\Q_X\otimes \sL_T)\cong H^0(T,\pH^i(Rf_*\Q_X)\otimes \sL_T).
%\]
%\[
%H^i(Y,\sL_{Y}) \cong H^i(T,R(f_Y)_*\Q_{Y}\otimes \sL_T)\cong H^0(T,\pH^i(R(j\circ f_Y)_*\Q_{Y})\otimes \sL_T).
%\]
\begin{equation}\label{eq:Y-Mellin}
\begin{split}
H^k(Y,\sL_{Y}) & 
\overset{\text{\ref{prop_iso1}}}{\cong} H^k(T,R(j\circ f_Y)_*\underline{\Q}_{Y}\otimes \sL_T)  \\
& \overset{\text{\ref{cor_GL}}}{\cong}
 \Mel(\pH^k(R(j\circ f_Y)_*\underline{\Q}_{Y})).  
%&\cong H^0(T,\pH^i(Rj_*R(f_Y)_*\Q_{Y}) .
\end{split}
\end{equation}
Similarly, $H^k(X,\sL_X)\cong \Mel(\pH^k(Rf_*\underline{\Q}_X))$. Consider the attaching triangle associated to $Rf_*\underline{\Q}_X$, for the embeddings $i$ and $j$:% $\C^* = b\sqcup \C^*\setminus b$:
\begin{equation}\label{atr}
i_!i^!Rf_*\underline{\Q}_X\to Rf_*\underline{\Q}_X \to Rj_*j^* Rf_*\underline{\Q}_X.
\end{equation}
Since  $j^*\circ Rf_*\cong R(f_Y)_*\circ \tilde\jmath^*$ (e.g., use \cite[Proposition 10.7(4)]{B08} and the fact that $j$ and $\tilde\jmath$ are open inclusions), we have:
\begin{equation}\label{eq22}
Rj_* j^* Rf_*\underline{\Q}_X\cong
Rj_*  R(f_Y)_*\tilde\jmath^*\underline{\Q}_X\cong
R(j\circ f_Y)_*\underline{\Q}_Y.
\end{equation}
Using (\ref{eq:Y-Mellin}), this complex computes (via the Mellin transformation) the Alexander modules of $Y$. Moreover, since $i_*\cong i_!$ is $t$-exact, we also have that 
\begin{equation}\label{eq23}
\pH^k(i_*i^!Rf_*\underline{\Q}_X) \cong i_*\pH^k(i^!Rf_*\underline{\Q}_X).\end{equation} Taking perverse cohomology of the attaching triangle \eqref{atr}, and using \eqref{eq22} and \eqref{eq23}, we obtain a long exact sequence:
\begin{equation}\label{eq:lesPerv}
 i_* \pH^k(i^!Rf_*\underline{\Q}_X)\to  \pH^k(Rf_*\underline{\Q}_X) \xrightarrow{\phi} \pH^k(R(j\circ f_Y)_*\underline{\Q}_{Y}) \to  i_* \pH^{k+1}(i^!Rf_*\underline{\Q}_X) .
\end{equation}
Denote the maximal smooth sub-object of $\pH^k(Rf_*\underline{\Q}_X)$ and $\pH^k(R(j\circ f_Y)_*\underline{\Q}_{Y})$ by $\sM_s^k$ and $\sN_s^k$ respectively. Since $\ker{\phi}$ is supported on $Z$ and $\sM_s^k$ is a (shifted) local system, by Lemma~\ref{lem:smoothPerverse}, we have that $(\ker{\phi})\cap \sM_s^k = 0$.
%
%In particular, $\Hom_{\Perv(T)}(\ker(\phi),\sM_s^k)=0$. 
%
Therefore, $\sM_s^k$ is a sub-object (via $\phi$) of $\pH^k(R(j\circ f_Y)_*\underline{\Q}_{Y})$. By maximality of $\sN_s^k$, we have $
\sM_s^k\subseteq \sN_s^k$, so (1) follows by applying the Mellin transformation and Proposition~\ref{prop_s0}.

\item Assume $f$ is a locally trivial fibration in a neighborhood $B$ of $Z$. In this case, $Rf_*\underline{\Q}_X$ is a complex whose cohomology sheaves are local systems on $B$. In particular, $\pH^k(Rf_*\underline{\Q}_X)$ is smooth on $B$ for all $k$ (see \cite[Proposition 4.3]{LMW20}).
Then, we can use Lemma~\ref{lem:smoothPerverse} again to conclude that the first map in (\ref{eq:lesPerv}) vanishes, yielding the short exact sequence:
\begin{equation}\label{eq:sesPerv}
0\to   \pH^k(Rf_*\underline{\Q}_X) \xrightarrow{\phi} \pH^k(R(j\circ f_Y)_*\underline{\Q}_{Y}) \to i_* \pH^{k+1}(i^!Rf_*\underline{\Q}_X) \to 0.
\end{equation}
Apply the exact functor $\Mel$ to \eqref{eq:sesPerv}, to obtain:
\[
0\to   \Mel(\pH^k(Rf_*\underline{\Q}_X)) \xrightarrow{\Mel(\phi)} \Mel(\pH^k(R(j\circ f_Y)_*\underline{\Q}_{Y})) \to \Mel(i_* \pH^{k+1}(i^!Rf_*\underline{\Q}_X)) \to 0.
\]
By Proposition~\ref{prop_s0}, the last term in this short exact sequence does not contain nonzero Artinian submodules. Therefore, the maximal Artinian submodules of the first two modules coincide.

\item Let us now assume that $Z=\{x\}$ is contained in a Zariski open subset $B$ over which $f$ is a fibration. Let us first compute $i^!Rf_*\underline{\Q}_X$. Factor $i$ as the composition
$$i:\{x\} \overset{i_x}{\hookrightarrow} B \overset{j_B}{\hookrightarrow} T.$$
Let $f_B\colon X_B=f^{-1}(B) \to B$ be the restriction of $f$ over $B$.
Then 
\begin{equation}\label{eq31} i^!Rf_*\underline{\Q}_X=(j_B \circ i_x)^! Rf_*\underline{\Q}_X \cong i_x^!j_B^!Rf_*\underline{\Q}_X\cong  i_x^!R(f_B)_*\underline{\Q}_{X_B}.\end{equation}
Since $f$ is a fibration over $B$, the complex $j_B^*Rf_*\underline{\Q}_X\cong R(f_B)_*\underline{\Q}_{X_B}$  is a locally constant complex on $B$, in the sense of \cite[Proposition 4.3]{LMW20}. It then follows from \cite[Proposition 3.7(b)]{B08} that 
\begin{equation}\label{eq32} i_x^!R(f_B)_*\underline{\Q}_{X_B}\cong i_x^*R(f_B)_*\underline{\Q}_{X_B}[-2n].\end{equation}

Using the fact that on a point space we have $\pH^k = \sH^k$, we get from \eqref{eq31} and \eqref{eq32} that:
\[
\pH^k(i^!Rf_*\underline{\Q}_X) \cong \sH^k(i^!Rf_*\underline{\Q}_X) \cong \sH^{k-2n}(R(f_B)_*\underline{\Q}_{X_B})|_x \cong H^{k-2n}(F_x,\Q).
\]
Now (\ref{eq:sesPerv}) becomes:
\[
0\to   \pH^k(Rf_*\underline{\Q}_X) \xrightarrow{\phi} \pH^k(R(j\circ f_Y)_*\underline{\Q}_{Y}) \to i_*H^{k-2n+1}(F_x,\Q) \to 0.
\]
Apply the exact functor $\Mel$, to obtain:
\[
0\to   \Mel(\pH^k(Rf_*\underline{\Q}_X)) \xrightarrow{\Mel(\phi)} \Mel(\pH^k(R(j\circ f_Y)_*\underline{\Q}_{Y})) \to \Mel(i_*H^{k-2n+1}(F_x,\Q)) \to 0.
\]
The last term in this short exact sequence is the free $A$-module $A\otimes_\Q H^{k-2n+1}(F_x,\Q)$. In particular, the sequence splits.

\item Let $f_B\colon X\to B$ be the restriction of the codomain of $f$, let $j_B$ be the open embedding $B\to T$. Let $x\in B$ and $F = f^{-1}(x)$. Since $f_B$ is a fibration,
\[
\pH^{k}(R(f_B)_*\underline{\Q}_X) \cong (R^{k-n}(f_B)_*\underline{\Q}_X)[n]
\]
is a shift of the local system with stalk $H^{k-n}(F,\Q)$ and monodromy induced by the monodromy acting on $F$. Consider $Rf_*\underline{\Q}_X \cong R(j_B)_*R(f_B)_*\underline{\Q}_X$. Since $B$ is a hypersurface complement, $R(j_B)_*$ is $t$-exact, and hence:
\[
\pH^k(R(j_B)_*R(f_B)_*\underline{\Q}_X)\cong 
R(j_B)_*\pH^k(R(f_B)_*\underline{\Q}_X) \cong 
R(j_B)_*(R^{k-n}(f_B)_*\underline{\Q}_X)[n].
\]
Note that $(j_B)^*$ (with a shift) induces an injection from the set of smooth sub-objects of $\pH^k(R(j_B)_*R(f_B)_*\underline{\Q}_X)$ and the set of local systems contained in $R^{k-n}(f_B)_*\underline{\Q}_X$ (its partial inverse is $(j_B)_{!*}$), which in turn injects into the set of subspaces of the stalk $R^{k-n}(f_B)_*\underline{\Q}_X|_x$. These injections preserve containments, so to find the maximal smooth sub-object $\sM_s^k$ of $\pH^k(R(j_B)_*R(f_B)_*\underline{\Q}_X)$ it is enough to find its stalk at $x$.

Using the stalk at $x$, we can think of local systems on $B$ (resp. on $T$) as $\pi_1(B,x)$-representations (resp. $\pi_1(T,x)$-representations). The functor $j_B^*$ is the pullback of the representation along $\eta\colon\pi_1(B)\twoheadrightarrow \pi_1(T)$. Therefore, a local system on $B$ comes from $T$ if and only $\ker \eta$ acts trivially, i.e. the largest sub-local system of $R(j_B)_*(R^{k-n}(f_B)_*\underline{\Q}_X)$ has stalk $\overline{H}{}^{k-n}$.

Finally, using Proposition~\ref{prop_s0}, the maximal Artinian submodule of $H^k(X,\sL_X)$ is $\Mel(\sM_s^k)$. The result follows from applying Lemma~\ref{lemma_ls}.\qedhere
\end{enumerate}
\end{proof}

\begin{proof}[Proof of Corollary~\ref{cor:homologyOpen}]
If $n=1$,  then $A\cong\Q[t^{\pm 1}]$ is a principal ideal domain, hence every finitely generated $A$-module (e.g., the Alexander modules appearing in this proof) decomposes into a direct sum of its free part and its torsion (maximal Artinian submodule) part.

%Part (1) follows by applying Proposition~\ref{prop:removeFibers} parts (1) and (2), and from the functoriality of the isomorphism in \cite[Proposition 2.4.1]{EGHMW}. Likewise, to prove part (2), we have that $ \Tors_A H_i(Y,\sL_Y)\to \Tors_A H_i(X,\sL_X)$ is an isomorphism.
Parts (1) and (2) follow by applying Proposition~\ref{prop:removeFibers}, parts (1), (2) and (3), and from the functoriality of the isomorphism in \cite[Proposition 2.4.1]{EGHMW}. Note that by the Universal Coefficients Theorem, the modules $H^i(Y,\sL_Y)$ and $H_i(Y,\sL_Y)$ have the same rank as $A$-modules, and similarly for $X$.

Part (3) follows from Proposition~\ref{prop:removeFibers} (4): we have that $\Tors_A H^{i+1}(X,\sL_X)$ is the largest submodule of $H^{i}(F,\Q)$ fixed by $K$. Using \cite[Proposition 2.4.1]{EGHMW}, $\Tors_A H_{i}(X,\sL_X)$ is the $\Q$-dual of $\Tors_A H^{i+1}(X,\sL_X)$ (and the isomorphism is compatible with the map to $H^i(F,\Q)$ and the monodromy action). Therefore, $\Tors_A H_{i}(X,\sL_X)$ is the largest quotient of $H_{i}(F,\Q)$ fixed by $K$.
\end{proof}

\section{A non-semisimple example}\label{sec:non-ss}
In this section, we construct a map $f: X\to \C^*$, where $X$ is a singular quasi-projective variety, such that the action of $\pi_1(\C^*)$ on $\Art H^2(X, \sL_X)$ is not semisimple. 

Let $B=\C^*\setminus \{1\}$. $X$ will be a family of nodal curves over $B$. Concretely, over $s\in B$, the fiber over $s$ will be $\P^1\setminus \{1,s\}$ with the points $0$ and $\infty$ identified. This example and its resulting variation of MHS was originally considered by Deligne in \cite[Section 13]{Deligne}.

Let $x,y$ be coordinates on $\C^2$ and $s$ be the coordinate on $B$. We define $X$ by
\[
X \coloneqq \{(x,y,s)\mid (sx-y)(y-x)+(s-1)^2x^2y=0\}\subset \C^2\times B,
\]
and $f:X\rightarrow \C^*$ is given by the projection onto the last coordinate.

Each non-empty fiber of $f$ (any fiber over $s\neq 1$) is a nodal cubic (in $\P^2$) with 2 smooth points removed (both of which lie on the line at infinity) %{\color{olive}I was supposed to change this sentence but I've forgotten what to}.  
If we use $\lambda$ for the coordinate of $\P^1$, $X$ is parametrized by:
\[
\begin{array}{crcl}
\Phi\colon & \P^1\times B\setminus (\{\lambda=1\}\cup \{\lambda=s\}) & \longrightarrow &X\subset \C^2\times B\\
& (\lambda,s) & \longmapsto &
\left(
\frac{\lambda}{(\lambda-1)(\lambda-s)},\frac{\lambda}{(\lambda-1)^2},s
\right)
\end{array}
\]
We use the affine coordinate $\lambda$, but $\Phi$ has the algebraic extension to $\lambda=\infty$, by letting $\Phi(\infty,s) = (0,0,s)$. It is straightforward to verify that the image of $\Phi$ is $X$, and that $\Phi$ induces a homeomorphism:
\[
\frac{\P^1\times B}{(0,s)\sim (\infty,s)} \cong X.
\]

%OLD DESCRIPTION AS A LOCALLY CLOSED SET IN THE PLANE
%. Consider a family of maps $\C\to \C^2$ parametrized by $s$ ($s\neq 1$), 
%\[
%t\mapsto \bigg( \Big(\frac{2(t-1)}{1-s}+1\Big)^2-1, \;\Big(\frac{2(t-1)}{1-s}+1\Big)^3-\Big(\frac{2(t-1)}{1-s}+1\Big)\bigg)
%\]
%As the total family of the map, we have a map $\Phi: \C\times (\C\setminus\{1\})\to \C^2\times (\C\setminus\{1\})$, defined by 
%\[
%(t, s)\mapsto \bigg( \Big(\frac{2(t-1)}{1-s}+1\Big)^2-1, \;\Big(\frac{2(t-1)}{1-s}+1\Big)^3-\Big(\frac{2(t-1)}{1-s}+1\Big), \;s\bigg).
%\]
%Let $X$ be the image of $(\C\setminus \{0\})\times (\C\setminus\{0, 1\})$ under the map $\Phi$. \textcolor{blue}{Notice that the third coordinate of $X$ is equal to $s$, which is always nonzero. Thus, we can define $f: X\to \C^*$ to be the projection of $X$ to the third coordinate.} Since $f^{-1}(s)$ is a closed nodal curve in $\C^2$ for any $s\in \C\setminus\{0, 1\})$, the image $\Phi(\C\times \C\setminus\{0, 1\})$ is closed in $\C^2\times \C\setminus\{0, 1\}$, and hence it is a quasi-projective variety. Notice that $X$ is open in $\Phi(\C\times \C\setminus\{0, 1\})$. Therefore, $X$ is also a quasi-projective variety. 
%
%Notice that the fiber of $f$ is also empty over $1$, and away from 1, it is isomorphic to the affine nodal curve $y^2=x^3+x^2$ with a smooth point removed. 

We claim that $S_0 H^2(X, \sL_X)$ is not semisimple. Based on the description above, $f$ is a locally trivial fibration over $B$, so we can apply Proposition~\ref{prop:removeFibers} (4): if $F$ is a generic fiber of $f$, $S_0H^2(X,\sL_X)$ is the subspace of $H^1(F,\Q)$ fixed by the kernel of $\pi_1(B)\to \pi_1(\C^*)$.

From the topological description of the fibers, we can compute the monodromy action on a basis of $H^1(F,\Q)\cong \Q^2$, or, equivalently, of $H_1(F,\Q)$. One such basis is shown in Figure~\ref{fig:fiber}. As $s\in B$ varies, the fiber $f^{-1}(s)$ varies in that one puncture moves to $\lambda=s$ and the other stays in place at $\lambda=1$. Let $\gamma_0,\gamma_1\in \pi_1(B)$ be loops going around the origin and $s=1$, respectively. Each of these induces a monodromy homeomorphism of $F$, namely the ones seen in Figures~\ref{fig:mon0} and~\ref{fig:mon1}. We can see that they induce the following automorphisms of $H_1(F,\Q)$.
\begin{align*}
\gamma_0&\colon a \mapsto a; &
\gamma_1&\colon a \mapsto a; \\
\gamma_0&\colon b \mapsto b-a ;&
\gamma_1&\colon b \mapsto b .
\end{align*}

\begin{figure}[h!]
\centering
\begin{minipage}{.33\textwidth}
  \centering
  \includegraphics[width=.95\linewidth]{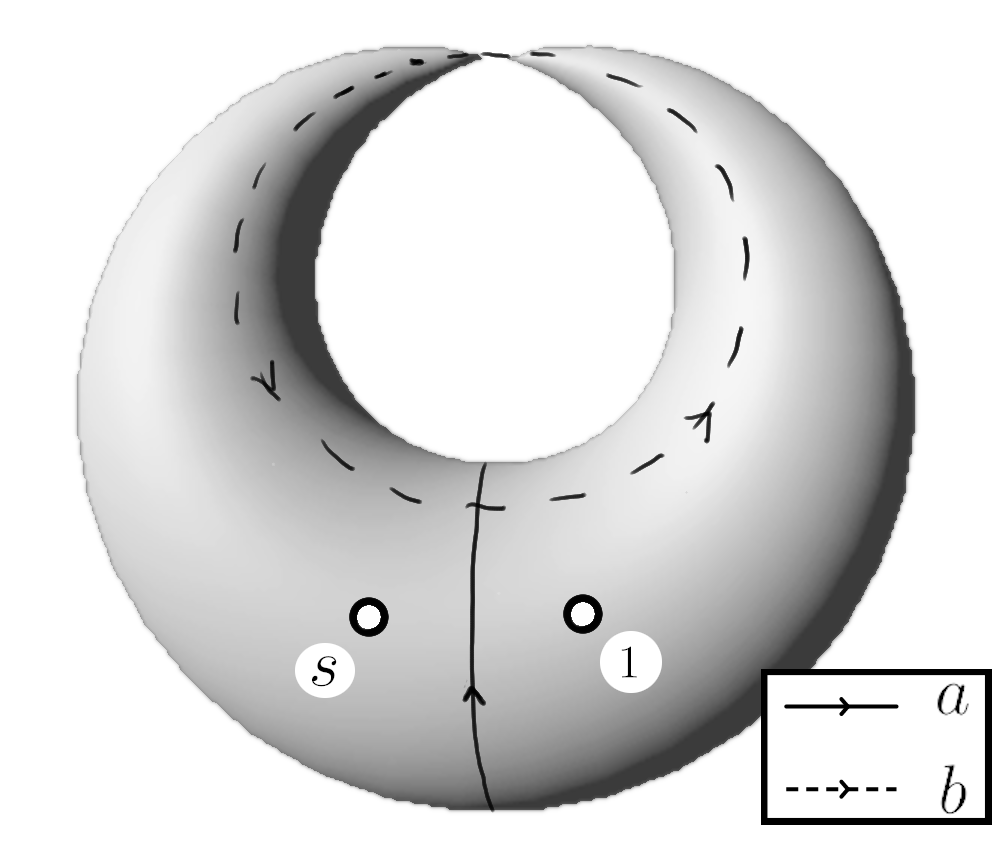}
  \captionof{figure}{The fiber of $f$.}\label{fig:fiber}
\end{minipage}%
\begin{minipage}{.33\textwidth}
  \centering
  \includegraphics[width=.95\linewidth]{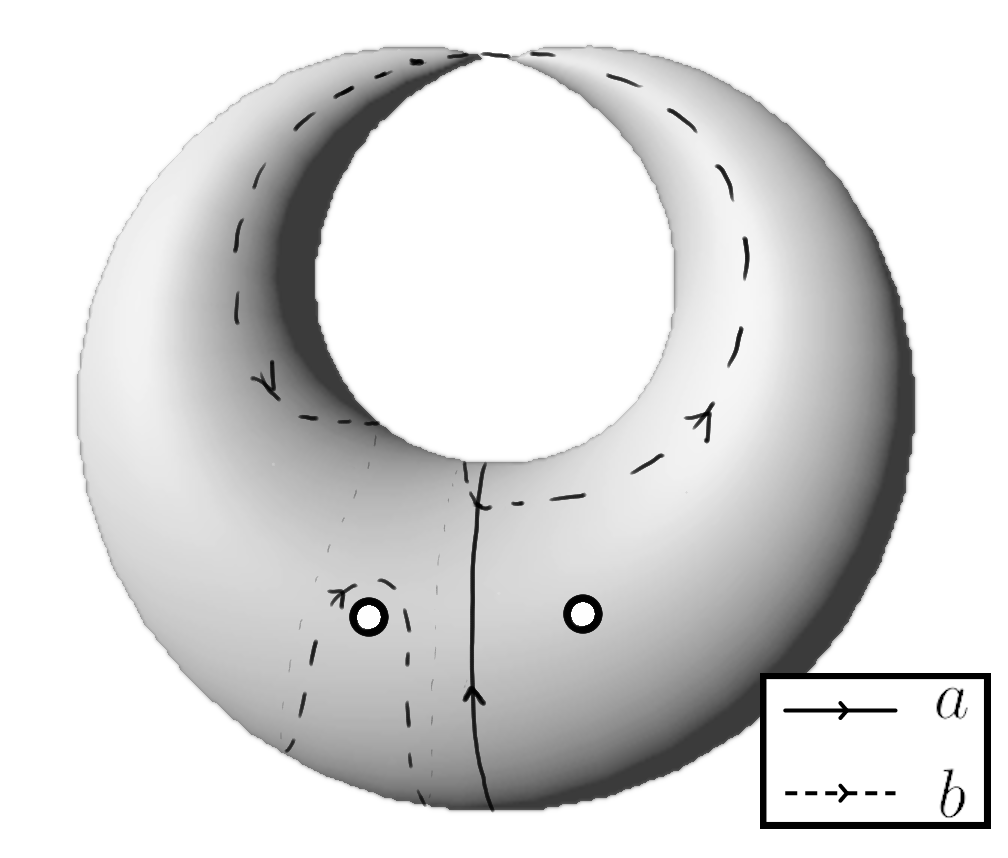}
  \captionof{figure}{The action of $\gamma_0$}\label{fig:mon0}
\end{minipage}%
\begin{minipage}{.33\textwidth}
  \centering
  \includegraphics[width=.95\linewidth]{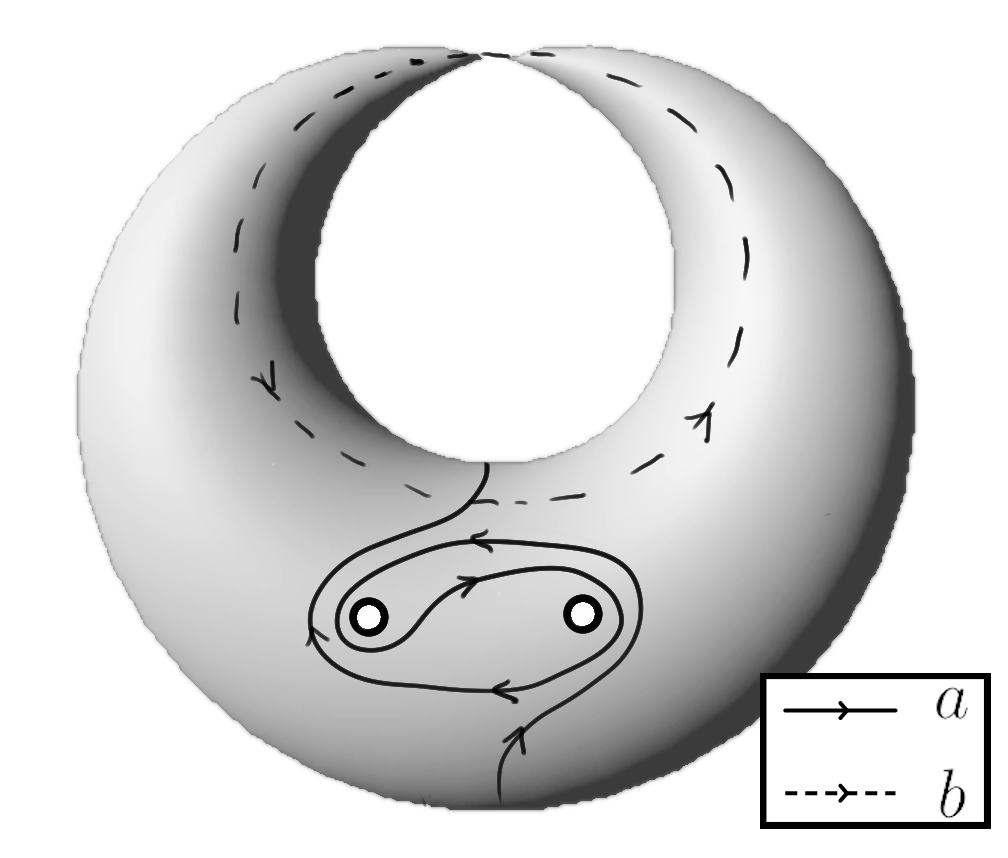}
  \captionof{figure}{The action of $\gamma_1$}\label{fig:mon1}
\end{minipage}
\end{figure}

Since $\gamma_1$ induces the identity, Proposition~\ref{prop:removeFibers} (4) tells us that $S_0H^2(X, \sL_X) = H^1(F,\Q)$, a 2 dimensional space. The module structure is induced by letting $t\in \Q[t^{\pm 1}]$ act as the monodromy of $\gamma_0$, which is a unipotent, but not semisimple, automorphism.

\begin{remark}
As pointed out in \cite{EGHMW}, we are not aware of non-semisimple examples of Alexander modules arising from an algebraic map $f:X \to \C^*$, with $X$ a smooth variety. However, A. Libgober \cite{Li21} has recently shown that such non-semisimple examples can be constructed if $f$ is replaced by a continuous map to $S^1$.
\end{remark}

%%%%%%%%%%%%%%%%%%%%%%%%%%%%%%%%%%%%%%%%%%%%%%%%%%%%%%%%%%%%

\end{document}